\documentclass[12pt]{amsart}

\usepackage[latin1]{inputenc}
\usepackage[all]{xy}
\usepackage{amsmath}
\usepackage{amssymb}
\usepackage{amsthm}
\usepackage{pdfsync}
\usepackage[linktocpage=true]{hyperref}
\usepackage{youngtab}

\newcommand{\coef}{\operatorname{coeff}}

\newtheorem*{proposition*}{Proposition}
\newenvironment{introprop}[1]{
	\\[.5\baselineskip]
	\noindent
	\textbf{Proposition #1.}\em
}{\\[0.25\baselineskip]}

\newenvironment{introtheorem}[1]{\hphantom{}
	\\[.5\baselineskip]
	\noindent
	\textbf{Theorem #1.}\em
}{\\[0.5\baselineskip]}

\newcommand{\stb}{,\ldots,}
\newcommand{\then}{\Rightarrow}

\makeatletter
\DeclareRobustCommand{\bigtimes}{%
	\mathop{\vphantom{\sum}\mathpalette\@bigtimes\relax}\slimits@}
\newcommand{\@bigtimes}[2]{\vcenter{\hbox{\make@bigtimes{#1}}}}
\newcommand{\make@bigtimes}[1]{%
	\sbox\z@{$\m@th#1\sum$}%
	\setlength{\unitlength}{\wd\z@}%
	\begin{picture}(1,1)
	\linethickness{.15ex}
	\Line(0.2,0.2)(0.8,0.8)
	\Line(0.2,0.8)(0.8,0.2)
	\end{picture}}

\newcommand{\GL}{\operatorname{GL}}
\newcommand{\Sp}{\operatorname{Sp}}
\newcommand{\SO}{\operatorname{SO}}

\newcommand{\U}{\operatorname{U}}

\newcommand{\Fl}{\operatorname{Fl}}
\newcommand{\fd}{\operatorname{fd}}

\newcommand{\Aut}{\operatorname{Aut}}
\newcommand{\Rep}{\operatorname{Rep}}
\newcommand{\rk}{\operatorname{rk}}
\newcommand{\Hom}{\operatorname{Hom}}

\newcommand{\iso}{\cong}

\newcommand{\im}{\operatorname{Im}}
\newcommand{\acsa}{\qquad \iff\qquad}
\newcommand{\id}{\operatorname{id}}
\newcommand{\OSP}{\operatorname{OSP}}
\newcommand{\bra}{\langle}
\newcommand{\braaa}{[\![}
\newcommand{\ket}{\rangle}
\newcommand{\kettt}{]\!]}
\newcommand*\conj[1]{\overline{#1}}

\newcommand{\codim}{\operatorname{codim}}

\newcommand{\Sq}{\operatorname{Sq}}
\newcommand{\inj}{\hookrightarrow}
\newcommand{\RP}{\mathbb{R}P}
\newcommand{\CP}{\mathbb{C}P}
\newcommand{\HP}{\mathbb{H}P}
\newcommand{\OP}{\mathbb{O}P}
\renewcommand{\O}{\mathbb{O}}

\newcommand{\PP}{\mathbb{P}}
\newcommand{\Gr}{\operatorname{Gr}}

\newcommand{\al}{\alpha}

\newcommand{\ga}{\gamma}
\newcommand{\Ga}{\Gamma}
\newcommand{\De}{\Delta}

\newcommand{\ka}{\kappa}
\newcommand{\la}{\lambda}

\newcommand{\si}{\sigma}
\newcommand{\Si}{\Sigma}
\newcommand{\Om}{\Omega}

\newcommand{\se}{\subseteq}
\newcommand{\su}{\backslash}

\newcommand{\Z}{\mathbb{Z}}
\newcommand{\R}{\mathbb{R}}
\newcommand{\C}{\mathbb{C}}
\newcommand{\N}{\mathbb{N}}
\newcommand{\Q}{\mathbb{Q}}
\newcommand{\F}{\mathbb{F}}

\newcommand{\D}{\mathcal{D}}

\newcommand{\hh}{\mathfrak{h}}
\newcommand{\HH}{\mathbb{H}}

\newtheorem{fact}{Fact}[section]
\newtheorem{lemma}[fact]{Lemma}
\newtheorem{theorem}[fact]{Theorem}
\newtheorem*{theorem*}{Theorem}
\newtheorem{defi}[fact]{Definition}
\newtheorem{exa}[fact]{Example}
\newtheorem{cla}[fact]{Claim}
\newtheorem*{sol}{\it Solution}
\newtheorem{rremark}[fact]{Remark}
\newtheorem{proposition}[fact]{Proposition}
\newtheorem{corollary}[fact]{Corollary}
\newtheorem{problem}{Problem}

\newenvironment{remark}{\begin{rremark} \rm}{\end{rremark}}
\newenvironment{definition}{\begin{defi} \rm}{\end{defi}}
\newenvironment{example}{\begin{exa} \rm}{\end{exa}}

\newtheorem{observation}[fact]{Observation}

\author{L\'aszl\'o M.\ Feh\'er}
\address{E\"otv\"os Lor\'and University, Budapest, Hungary}
\email{lfeher@cs.elte.hu}

\author{\'Akos K.\ Matszangosz}
\address{Alfr\'ed R\'enyi Institute of Mathematics, Budapest, Hungary}
\email{matszangosz.akos@gmail.com}


\thanks{ L.\ M.\ F. was  partially supported by NKFI 112703 and 112735 as well as ERC Advanced Grant LTDBud and enjoyed the hospitality of the R\'enyi Institute. \\
A.\ K.\ M.\ is supported by the Hungarian National Research, Development and Innovation Office, NKFIH K 119934.}

\title{Halving spaces and lower bounds in real enumerative geometry}
\begin{document}
	\maketitle
\begin{abstract} We develop the theory of \emph{halving spaces} to obtain lower bounds in real enumerative geometry. Halving spaces are topological spaces with an action of a Lie group $\Ga$ with additional cohomological properties. For $\Ga=\Z_2$ we recover the conjugation spaces of Hausmann, Holm and Puppe. For $\Ga=\U(1)$ we obtain the \emph{circle spaces}. We show that real even and quaternionic partial flag manifolds are circle spaces leading to non-trivial lower bounds for even real and quaternionic Schubert problems. To prove that a given space is a halving space, we generalize results of Borel and Haefliger on the cohomology classes of real subvarieties and their complexifications. The novelty is that we are able to obtain results in rational cohomology instead of modulo 2. The equivariant extension of the theory of circle spaces leads to generalizations of the results of Borel and Haefliger on Thom polynomials.
\end{abstract}

	\tableofcontents
\section{Introduction}\label{sec:intro}

The answer to an enumerative geometry problem over $\C$ is a natural number. In the case of Schubert calculus, this number is completely determined by the rational cohomology ring of the Grassmannian.

For an enumerative geometry problem over $\R$, the space of generic configurations is no longer connected, so the answer will be a finite list of natural numbers. The number of solutions for  the corresponding complex problem is an obvious upper bound for the number of real solutions. Moreover, modulo two the number of real solutions should be congruent to the number of complex solutions (since the non-real solutions come in complex conjugate pairs).

Lower bounds are more difficult to obtain, for example see \cite{Welschinger2005}, \cite{SoprunovaSottile2006}. Several authors noticed (e.g. \cite{finashin_abundance_2012}, \cite{OkonekTeleman2014signs}) that for some problems a cohomological calculation is available. For example consider the following:
\begin{problem} How many 4-dimensional linear subspaces of $\R^8$ intersect four given generic 4-dimensional linear subspaces in 2 dimensions?	
\end{problem}
The number of solutions for the complex problem is the integral
\[ \int_{\Gr_4(\C^8)}[\si_{(2,2)}]^4,\]
where $\si_{(2,2)}$ is the Schubert variety corresponding to the partition $(2,2)$, and $[\si_{(2,2)}]\in H^8(\Gr_4(\C^8))$ is the cohomology class represented by it. According to the Schubert calculus the answer is 6, which is an upper bound for the number of solutions for the real problem.  The key to find a lower bound is the observation that there is a degree halving isomorphism $\kappa$ between the rational cohomology ring of the real Grassmannian $\Gr_{4}(\R^{8})$ and the rational cohomology ring of the complex Grassmannian $\Gr_{2}(\C^{4})$. Indeed, the first ring is generated by the Pontryagin classes of the tautological subbundle, the second is generated by the Chern classes of the tautological subbundle, and the assignment $p_i\mapsto c_i$ extends to an isomorphism. Moreover, the real Schubert variety $\si^\R_{(2,2)}$ is a cycle (it represents a cohomology class in $\Gr_{4}(\R^{8})$), and $\kappa[\si^\R_{(2,2)}]=[\si_{(1,1)}]$. This implies that
\[ \int_{\Gr_4(\R^8)}[\si^\R_{(2,2)}]^4=\int_{\Gr_2(\C^4)}[\si_{(1,1)}]^4=2.\]
In the absence of a complex structure on $\Gr_4(\R^8)$ this result means only that the \emph{signed sum} of the solutions is 2, therefore 2 is a lower bound to the number of solutions.

Our main goal here is to study how general this situation is. It turns out that even Grassmannians $\Gr_{2k}(\R^{2n})$---and more generally even partial flag manifolds---possess a structure we will call \emph{circle space}, which allows us to get lower bounds for a large family of real enumerative problems. Circle spaces are closely related to conjugation spaces, which were introduced  by Hausmann, Holm and Puppe \cite{HausmannHolmPuppe2005} to understand the topological background of the theory of Borel and Haefliger \cite{BorelHaefliger1961}.

The object of study in \cite{BorelHaefliger1961} is a nonsingular projective variety $X$ over $\C$, which is the complexification of its real part. Let $\Ga=\Z_2$ act on $X$ by complex conjugation and let $X^\Ga$ denote the real part (the $\Ga$-fixed point set). They show in Proposition 5.15 of \cite{BorelHaefliger1961} that assuming that all $x\in H^*(X;\F_2)$ and $y\in H^*(X^\Ga;\F_2)$ can be represented by algebraic cycles defined over $\R$ and real algebraic cycles respectively, we have a degree halving isomorphism $\kappa$ between the mod 2 cohomology of $X$ and the mod 2 cohomology of the real part $X^\Ga$. This isomorphism has the property $\kappa[Z]=[Z^\Ga]$ for any algebraic cycle $Z$,  which  is the complexification of its real part $Z^\Ga$.

Conjugation spaces are spaces with a $\Gamma=\Z_2$-action, satisfying a certain condition in mod 2 cohomology, called the \emph{restriction equation}. The nonsingular projective varieties of above with the conjugation action are the main examples of conjugation spaces.

In this paper we extend the theory of conjugation spaces to \emph{halving spaces} (Definition \ref{def:halvingspace}), where the acting group $\Gamma$ is allowed to be $\U(1)$ or $\Sp(1)$, and the conjugation equation becomes the analogous condition in rational cohomology. The existence of a degree halving ring isomorphism $\kappa$ between the rational cohomology ring of the fixed point set $X^\Ga$ and of the halving space $X$ can be established analogously to the case of conjugation spaces (Theorem \ref{thm:multiplicativity}). Our main tool to find examples of halving spaces is the \emph{generalized Borel-Haefliger theorem} (Theorem \ref{thm:BHI}), which has the following special case:
\begin{introtheorem}{\ref{thm:BHI}}
	Let $\Ga=\U(1)$ and let $X$ be a compact orientable $\Ga$-manifold, whose rational cohomology groups are additively generated by good $\U(1)$-invariant cycles $[Z_i]$ (Definition \ref{def:goodcycle}), such that $\codim_\R Z_i=2\codim_\R Z_i^\Ga$. Assume that $X^\Ga$ is connected. Then $X$ is a halving space, and the assignment sending $[Z_i]$ to $[Z_i^\Ga]$ determines a degree-halving multiplicative isomorphism between $H^{2*}(X;\Q)$ and $H^*(X^\Ga;\Q)$.
\end{introtheorem}
The proof of Theorem \ref{thm:BHI} is an adaptation of Van Hamel's work \cite{VanHamel2007}. Using Theorem \ref{thm:BHI} we obtain
\begin{introtheorem}{\ref{thm:doubleflagcirclespace}} Let $\Ga=\U(1)$ and let $\Ga$ act on $\Fl_{2\D}(\R^{2n})$ obtained by identifying $\C^n$ with $\R^{2n}$. With this action it is a circle space, with fixed point set $\Fl_\D(\C^n)$. Furthermore, the degree-halving ring isomorphism $\ka$ associated to the circle space structure satisfies:
$$\ka[\si_{DI}^\R]=2^{|I|}[\si_I^\C],$$
where $[\si_I^\C]\in H^{2|I|}(\Fl_\D(\C^N))$.
\end{introtheorem}
and
\begin{introtheorem}{\ref{thm:quatflagcirclespace}}
	With the $\Ga=\U(1)$-action defined by inner automorphisms (see Section \ref{subsec:Galois}), $\Fl_\D(\HH^n)$ is a circle space, with fixed point set $\Fl_\D(\C^n)$. Furthermore, the degree-halving ring isomorphism $\ka$ associated to the circle space structure satisfies:
	$$ \ka[\si_I^\HH]=2^{|I|}[\si_I^\C],$$
	where $[\si_I^\C]\in H^{2|I|}(\Fl_\D(\C^n))$.
\end{introtheorem}
 Using products and connected sums we can create further examples.  Halving spaces with $\Ga=\Sp(1)$ are less common, our examples are the octonionic flag manifolds (Theorem \ref{thm:octoflagcirclespace}).

Circle spaces allow us to study various real (and quaternionic) enumerative problems (the \emph{double Schubert problems}, see Section \ref{sec:applications}) and finding lower bounds for them. A different example is
\begin{introprop}{\ref{32}}
	Given four generic linear maps $\alpha_i:\R^8\to\R^8$  the number of 4-dimensional subspaces $V$ satisfying $\dim(V\cap\alpha_i(V))=2$ for $i=1,\dots,4$ is at least 32.
\end{introprop}
This example is connected with quivers and leads naturally to an equivariant variation of the theory, giving a partial analogue (Theorem \ref{thm:realquiver}) of Borel and Haefliger's result on mod 2 Thom polynomials \cite[Theorem 6.2.]{BorelHaefliger1961}.
\bigskip

\textbf{Structure of the paper.} In Section \ref{sec:halving}, we generalize the definition of conjugation spaces \cite{HausmannHolmPuppe2005} to $\U(1)$ and $\Sp(1)$-actions and give the first example of a circle space.

Section \ref{sec:geometry} contains the key technical results. We introduce the notion of a halving cycle and show that if $Z\subset X$ is a  halving cycle, then $\kappa[Z]=\mu_Z [Z^\Ga]$, where $\mu_Z\in \N$ is the \emph{excess multiplicity} of $Z$. Our main tool is the excess intersection formula of Quillen.

The main result of Section \ref{sec:BHI} is the  generalized Borel-Haefliger theorem (Theorem \ref{thm:BHI}): we give a sufficient condition (existence of a basis in cohomology, represented by halving cycles) for a $\Ga$-manifold to be a halving space.

We give examples of halving spaces in Section \ref{sec:examples} using  the  generalized Borel-Haefliger theorem. We give a description of the Schubert calculus of the even real and the quaternionic flag manifolds.

In Section \ref{sec:applications} we give some applications of these results in real and quaternionic enumerative geometry using Schubert calculus.

In Section \ref{sec:further} we study a quiver type degeneracy locus and  related enumerative problems. In conclusion we mention an equivariant extension the generalized Borel-Haefliger theorem.

\textbf{Acknowledgment.} We are grateful to Bal\'azs Csik\'os, Matthias Franz, Tara Holm, Liviu Mare, Rich\'ard Rim\'anyi, Endre Szab\'o and Andrzej Weber for several helpful discussions on the subject of this paper.

\section{Halving spaces}\label{sec:halving}
In this section we extend the definition of conjugation spaces \cite{HausmannHolmPuppe2005} to $\Ga=\U(1)$ and $\Sp(1)$-actions and cohomology with $\Z$ or $\Q$-coefficients. Most of the proofs of Hausmann, Holm and Puppe \cite{HausmannHolmPuppe2005} generalize word by word to halving spaces, we give a discussion for the sake of completeness.
\subsection{Preliminaries}
\begin{definition}\index{halving pair}
	Let $\Ga$ be a Lie group, and $R$ be a ring. We say that $(\Ga,R)$ is a \emph{halving pair}\label{word:halvingpair}, if $H^*(B\Ga;R)\iso R[u]$ with $u\in H^D(B\Ga;R)$, for some $D\in\N_+$.
\end{definition}
\begin{example}[See also Steenrod's polynomial realization problem \cite{Steenrod1961}]
	\begin{itemize}
		\item 	$(\Z_2,\F_2)$ is a halving pair with $D=1$, since $H^*(B\Z_2;\F_2)\iso \F_2[u]$ for $u=w_1(S)$ where $S\to \RP^\infty$ is the tautological bundle.
		\item $(\U(1),\Z)$ is a halving pair with $D=2$, since $H^*(B\U(1);\Z)\iso \Z[u]$ for $u=c_1(S)$ where $S\to \CP^\infty$ is the tautological bundle.
		\item $(\Sp(1),\Z)$ is a halving pair with $D=4$, since $H^*(B\Sp(1);\Z)\iso \Z[u]$, for $u=q_1(S)$ where $S\to \HP^\infty$ is the tautological bundle and $q_i$ denotes the quaternionic Pontryagin class.
	\end{itemize}
\end{example}
Throughout the paper, we will always take cohomology with coefficients $R$ corresponding to the group action $\Ga$ for each halving pair $$(\Ga,R)\in \{(\Z_2,\F_2),(\U(1),\Q), (\Sp(1),\Q)\},$$
 and $u\in H^D(B\Ga;R)$ is always defined as above. To simplify the discussion (e.g.\ for localization theorems) we work with field coefficients, although some of the discussion generalizes to $R=\Z$. If the coefficient ring $R$ is clear from the context, then we drop it from the notation, i.e. $H^*(X)=H^*(X;R)$.

Let $X$ be a $\Ga$-space, i.e.\ an action of $\Ga$ is given on $X$, and let us fix the coefficient ring $R$. Let $H_\Ga^*(X;R)=H^*(B_\Ga X;R)$ denote the \emph{$\Ga$-equivariant cohomology of $X$}, where $B_\Ga X=E\Ga\times_\Ga X$ is the Borel construction; let $\rho:H_\Ga^*(X;R)\to H^*(X;R)$ denote the forgetful map. A graded $R$-module homomorphism $\si:H^*(X)\to H_\Ga^*(X)$ is called a \emph{Leray-Hirsch section}, if $\rho\si=\id$. We call $\si$ a Leray-Hirsch section, since its existence implies that the condition of the Leray-Hirsch theorem is satisfied for the fiber bundle $X\to E\Ga\times_\Ga X\to B\Ga$.

We have a restriction map $r:H_\Ga^*(X)\to H^*(X^\Ga)[u]$, where we use that the $\Ga$-action on $X^\Ga$ is trivial, so $H^*_\Ga(X^\Ga)= H^*(X^\Ga)[u]$.
\subsection{The definition of halving spaces}
By adapting Hausmann, Holm and Puppe's definition of conjugation spaces \cite{HausmannHolmPuppe2005} to halving pairs, we obtain the central definition of this paper:
\begin{definition}\label{def:halvingspace}
	Let $X$ be a $\Ga$-space, and denote by $X^\Ga$ its fixed point set. Assume $X$ has nonzero cohomology only in $2Di$ degrees, and that there exists a Leray-Hirsch section $\si:H^*(X)\to H_\Ga^*(X)$ satisfying the following \emph{degree condition}\index{degree condition!DC}\label{word:degreecondition}:
	\begin{itemize}
		\item[(DC)]: $r(\si(x))$ is a polynomial in $u$ of degree \emph{exactly} $i$ for all $x\in H^{2Di}(X)$.
	\end{itemize}
	A $\Ga$-space $X$ satisfying these conditions is called a \emph{halving space}\label{word:halvingspace}.
\end{definition}
Let us unravel this definition. The halving space structure involves the following maps, which will be used in the following:
\begin{displaymath}
\xymatrix{			
	H_\Ga^*(X)\ar@{->}[r]^-{r}\ar@{->}[d]^{\rho}& H^*(X^\Ga)[u] \ar@{->}[d]\\
	H^*(X)\ar@{->}[r] \ar@/^1pc/{-}^{\si}&H^*(X^\Ga)}
\end{displaymath}

Let
\begin{equation}\label{eq:kadef}
\ka:H^{2*}(X)\to H^*(X^\Ga)
\end{equation} be the degree halving $R$-module homomorphism $\ka(x):=\coef(r\si(x),u^i)$ for $x\in H^{2Di}(X)$. The pair $(\ka,\si)$ is called a \emph{cohomology frame}\label{word:cohframe}.

With this notation, the degree condition (DC) means that $\ka$ is injective and that for $x\in H^{2Di}(X)$
\begin{equation}\label{eq:restrictioneqn}
r\si(x)=\ka(x)u^i+\la_1 u^{i-1}+\ldots +\la_{i-1}u+\la_i,
\end{equation}
where $\la_j\in H^{D(i+j)}(X^\Ga)$. We call equation \eqref{eq:restrictioneqn} \emph{restriction equation}\index{restriction equation}, it is called \emph{conjugation equation} in \cite{HausmannHolmPuppe2005}. For $(\Ga,R)=(\Z_2,\F_2)$, the definition of halving spaces is the same as the definition of conjugation spaces in \cite{HausmannHolmPuppe2005}, except that we don't require $\ka$ to be surjective.\\

It would be more precise to call a halving space a $(\Ga,R)$-halving space, however when $(\Ga,R)$ is fixed, we simply say $X$ is a halving space. We consider halving spaces for the following halving pairs $(\Ga,R)$:
\begin{itemize}
	\item Hausmann-Holm-Puppe's conjugation spaces\index{conjugation space} \cite{HausmannHolmPuppe2005}\label{word:conjugationspace} for the halving pair $(\Z_2,\F_2)$.
	
	\item \emph{Circle spaces}\index{circle space}\label{word:circlespace} for the halving pair $(\U(1),\Q)$. This is the main case we will consider.
	
	\item \emph{Quaternionic halving spaces} for the halving pair $(\Sp(1),\Q)$.\index{quaternionic halving space}\label{word:quathalvingspace}
\end{itemize}	

\subsection{Main properties}
To motivate the following discussion, we list some of the nice properties of halving spaces:
\begin{itemize}
	\item $\ka$ is a degree-halving ring homomorphism,
	\item $\si$ is a ring homomorphism, therefore the Leray-Hirsch isomorphism induced by $\si$ is a $H_\Ga^*=H^*(B\Ga;R)$-algebra isomorphism,
	\item the cohomology frame $(\ka,\si)$ is unique.
\end{itemize}

The proof of these properties relies on the following lemma, which is implicitly used in \cite{HausmannHolmPuppe2005}, its proof is the same, we repeat it for the sake of completeness.

\begin{lemma}[Degree Lemma]\label{lemma:degreelemma}
	Let $X$ be a halving space with cohomology frame $(\ka,\si)$. Let $D$ denote the degree of the generator $u\in H_\Ga^*$. Then for $x\in H^{2Dk}_\Ga(X;R)$ $$x\in \im\si\iff \deg_u(rx)=k$$
\end{lemma}
\begin{proof}
	The direction $\then$ holds by definition. For the other direction, let $x\in H^{2Dk}_\Ga(X;R)$, and assume $x\not\in \im\si$. By the Leray-Hirsch theorem
	$$ x=\sum_{i=0}^k\si(\xi_i)u^{2(k-i)}$$
	for some $\xi_i\in H^{2Di}(X)$. Since $r$ is an $R[u]$-module morphism,
	$$ rx=\sum_{i=0}^k(r\si\xi_i)u^{2(k-i)}=\sum_{i=0}^kp_i(u)u^{2(k-i)}$$
	where $r\si\xi_i=p_i(u)\in H^*(X^\Ga)[u]$ is a polynomial in $u$ of degree $\leq i$. Then $p_i(u)u^{2(k-i)}$ has degree $\leq 2k-i$ for each $i$. Take the smallest $0\leq i<k$, such that $\xi_i\neq 0$. Then $r\si\xi_i=\ka(\xi_i) u^i+...$ and since $\ka$ is injective, $r\si\xi_i$ has degree $i$. It follows that $rx$ is a polynomial in $u$ of degree $2k-i$. Since $i<k$, this is a contradiction, since $rx$ is a polynomial of degree $k$ by assumption. Therefore $\xi_i=0$ for $i<k$ and $x=\si(\xi_k)$.
	\end{proof}

In particular, by using the Leray-Hirsch theorem, Lemma \ref{lemma:degreelemma} implies that if $x\in H^{2Dk}_\Ga(X;R)$, and $\deg_u(rx)<k$, then $x=0$.
\begin{theorem}\label{thm:multiplicativity}
	Let $X$ be a halving space. Then $\ka$ and $\si$ are multiplicative.
\end{theorem}
\begin{proof}
	Let $a\in H^{2Dk}(X)$, $b\in H^{2Dl}(X)$. Set $x:=\si(a)\si(b)$, note that $\rho(x)=ab$. Then
     $$rx=r\big(\si(a)\big)r\big(\si(b)\big)=(\ka(a)u^k+...)(\ka(b)u^l+...),$$
	so $x=\si(y)$ for some $y\in H^{2D(k+l)}(X)$ by the degree lemma. Since $$ab=\rho(x)=\rho(\si(y))=y,$$
	so $x=\si(y)=\si(ab)$ and by definition $x=\si(a)\si(b)$ proving multiplicativity of $\si$. Using
     $$r\si(ab)=r\big(\si(a)\big)r\big(\si(b)\big),$$
the degree $k+l$ part of the left hand side is $\ka(ab)$ and on the right hand side $\ka(a)\ka(b)$.
\end{proof}
In particular, multiplicativity of $\si$ implies that the Leray-Hirsch isomorphism induced by $\si$
$$ H_\Ga^*(X)\iso H^*(X)\otimes_R H_\Ga^*$$
is a ring isomorphism.
\begin{corollary}[Naturality]\label{cor:naturality}
	{Let $X, Y$ be halving spaces with some cohomology frames $(\ka_X,\si_X)$ and $(\ka_Y,\si_Y)$ for $X$ and $Y$ respectively. If $f:X\to Y$ is a $\Ga$-equivariant map, then}
	$$ \si_X \circ H^*f=H_\Ga^*f \circ \si_Y:H^{*}(Y)\to H_\Ga^*(X)$$
	and
	$$ \ka_X \circ H^*f=H^*f^\Ga \circ \ka_Y:H^{2*}(Y)\to H^*(X^\Ga)$$
\end{corollary}
The proof proceeds similarly as the proof of multiplicativity using the degree lemma. Note that naturality implies uniqueness of the cohomology frame $(\ka,\si)$.

\begin{example}\label{ex:realGrassmannian}
	Let $X=\Gr_{2}(\R^{2n})$. Consider the $\Ga=\U(1)$-action on $\R^{2n}$ by identifying it with $\C^n$ and acting by complex multiplication. This induces an action on $X$. With this action, $X$ is a circle space with fixed point set $X^\Ga=\CP^{n-1}$.
\end{example}
\begin{proof}
	Since the $\U(1)$-invariant subspaces are exactly the complex subspaces, $X^\Ga$ can be identified with $\Gr_1(\C^n)=\CP^{n-1}$. In terms of characteristic classes, the ring structure can be written as
	$$ H^*(\Gr_2(\R^{2n});\Q)=\Q[x]/x^n,\qquad H^*(\CP^{n-1};\Q)=\Q[y]/y^n,$$
where $x=p_1(S_\R)$, $y=c_1(S_\C)$, and $S_\R\to \Gr_{2}(\R^{2n})$, $S_\C\to \CP^{n-1}$ are the tautological bundles.

	Let $\si$ be defined on the additive generators by $\si(x^i):=(p_1^\Ga(S_\R)-u^2)^i$, where $p_i^\Ga(S_\R)=p_i(B_\Ga S_\R\to B_\Ga X)$ denotes the $i$-th equivariant Pontryagin class. This $\si$ is a Leray-Hirsch section, and it satisfies the degree condition (DC), which can be shown by the following computation. First,
	$$ B_\Ga(S_\R)|_{X^\Ga}=B_\Ga(S_\R|_{X^\Ga})=B_\Ga S_\C.$$
Since $\Ga$ acts on $S_\C$ by complex multiplication, we can rewrite it as the tensor product of equivariant bundles $S_\C=S_\C^0\otimes_\C \C^{tw}$, where $S_\C^0$ denotes $S_\C$ with the trivial action and $\C^{tw}$ denotes the trivial bundle, with the nontrivial $\Ga$-action given by complex multiplication. Then as bundles over $B_\Ga X^\Ga=B\Ga \times X^\Ga$,
	$$B_\Ga S_\C=S_\C^0\otimes_\C \tau,$$ where $B_\Ga \C^{tw}=\tau\to \CP^\infty$ is the tautological bundle and we omit the notation for the pullbacks to the product. Therefore
	$$rp_{i}^\Ga(S_\R)=p_i(B_\Ga S_\C)=(-1)^ic_{2i}((S_\C^0\otimes_\C \tau)\otimes_\R \C)$$
	and
	$$c_{*}((S_\C^0\otimes_\C \tau)\otimes_\R \C)=c_{*}(S_\C^0\otimes_\C \tau)c_*(\conj{S_\C^0\otimes_\C \tau})=(1+y+u)(1-y-u),$$
	so
	$$ r\si(x^i)=r(p_1^\Ga(S_\R)-u^2)^i=((y+u)^2-u^2)^i=(2yu+y^2)^i,$$
	hence $\si$ satisfies the degree condition, and $(\ka,\si)$ is a cohomology frame with $\ka(x^i)=2^iy^i$.
\end{proof}

Similar explicit formulas can be given for even partial flag manifolds, however we will use Theorem \ref{thm:BHI}. to prove that they are circle spaces.
\section{Geometry}\label{sec:geometry}
 The original Borel-Haefliger theorem states that for a smooth complexified projective variety $X_\C$, under certain conditions, the complexification map $[Z_\R]\mapsto [Z_\C]$ from $H^*(X_\R;\F_2)\to H^{2*}(X_\C;\F_2)$ is a well defined multiplicative isomorphism, where $Z_\R$ is a real subvariety.
Here by \emph{real subvariety}, we mean a subset $Z_\R\se X_\R$ defined by real algebraic equations whose Zariski closure in $X_\C$ equals the set of complex points defined by the same equations -- its \emph{complexification}. Note, that $Z_\R=(Z_\C)^\Ga$ is the fixed point set of the complex conjugation on $Z_\C$. In terms of conjugation spaces, we can rephrase the Borel-Haefliger theorem as
\begin{theorem}\label{pr:conj-ew}\cite{BorelHaefliger1961}, \cite{VanHamel2007}
	For $Z_\C\se X_\C$ with the properties above: $X_\C$ is a conjugation space and
 \begin{equation}\label{eq:kappa}
 \kappa[Z_\C]=[Z_\R].
 \end{equation}
\end{theorem}

We would like to find a similar connection for halving spaces. We illustrate on the Schubert varieties of Example \ref{ex:realGrassmannian}, that some adjustments are necessary.

Fix a complete flag $F_\bullet$ in $\R^{2n}$ such that $F_{2i}$ is $\U(1)$-invariant for all $i$; then
$$F_\bullet^\C=(F_0\leq F_2\leq \ldots\leq F_{2n})$$ is a complex flag, $F_i^\C=F_{2i}$. Let $Z$ be the real Schubert variety  $Z=\si_{(2i,2i)}^\R(F_\bullet)$. This $Z$ is $\U(1)$-invariant and naturally isomorpic to the Grassmannian $\Gr_2(F_{2(n-i)})$. The fixed point set $Z^\Ga$ is the set of complex lines contained in $F_{n-i}^\C$, so $Z^\Ga=\si_i^\C(F^\C_\bullet)$, which is naturally isomorphic to the projective space $\PP F_{n-i}^\C$.
If $Q_\R\to \Gr_{2}(\R^{2n})$ and $Q_\C\to \CP^{n-1}$ denote the tautological quotient bundles, then
	\begin{equation}\label{eq:SchubertPontryagin}
	[\si_{(2i,2i)}^\R]=p_i(Q_\R)=(-x)^i,\qquad  [\si_i^\C]=c_i(Q_\C)=(-y)^i,
	\end{equation}
with the appropriate coorientation of the submanifold $\si_{(2i,2i)}^\R\subset\Gr_2(\R^{2n})$,
	implying that instead of \eqref{eq:kappa} we have $\ka[Z]=2^i[Z^\Ga]$.

The goal of this section is to prove that for reasonable 'subvarieties' $Z$ of the halving space $X$ we have $\ka[Z]=\mu_Z[Z^\Ga]$, where $\mu_Z\in R$ is the \emph{excess multiplicity} of $Z$. First we discuss the notion of topological varieties and their fundamental cohomology classes in Section \ref{subsubsec:topvar}, then review a part of excess intersection theory necessary to introduce the notion of excess multiplicities in Section \ref{sec:eif}.

In this section we restrict our attention to \emph{halving manifolds}, smooth manifolds which are halving spaces with smooth $\Ga$-action with $(\Ga,R)$ being $(\Z_2,\F_2)$ or $(\U(1),\Q)$ and fix the corresponding degree $D=1,2$. In the context of conjugation manifolds, many of their properties can be found in \cite[Section 2.7]{HambletonHausmann2011}.

 Most of the results also hold for $(\Sp(1),\Q)$ under additional assumptions, see Remarks \ref{rmk:Sp1halvingcycle} and \ref{rmk:injectivitylemma} iii).

\subsection{Fundamental classes of real varieties}
The Borel-Haefliger theorems involve cohomology classes of real and complex algebraic varieties, and there are slightly different approaches in defining these. We use the definitions of \emph{topological varieties} of the original Borel-Haefliger paper \cite{BorelHaefliger1961} and van Hamel \cite{VanHamel2007}. More precisely, we use a variation, in the sense that we work in cohomology instead of homology; in particular we use coorientability instead of orientability. By manifold and submanifold we always mean \emph{smooth} manifold and submanifold. We expect that with enough care, most of the discussion generalizes to topological, even cohomological manifolds, however we will not need this generality.

This section is standard and well-known to experts, however because of the slight variations of these notions, we found it reasonable to include some details at least to fix terminology.

\subsubsection{Topological varieties}\label{subsubsec:topvar}
Borel and Haefliger define fundamental classes for a class of topological spaces that we will call---following van Hamel \cite{VanHamel2007}---\emph{topological varieties} (or in Borel and Haefliger's notation $VS_n$ spaces). 
This is a class of topological spaces which includes analytic manifolds and algebraic varieties; both real and complex.

Throughout the discussion, fix a smooth ambient manifold $X$ (connected, not necessarily orientable) and a principal ideal domain $K$ (typically $K=\Z$ or a field $\F_p, \Q, \R$), and all cohomology is taken with $K$-coefficients.
\begin{definition}\label{def:cohdim}
	$\Si\se X$ has \emph{cohomological codimension\index{cohomological codimension} $\codim_K \Si\geq k$} if $H^i(X, X\su \Si)=0$ for all $i\leq k-1$. It has \emph{cohomological codimension $\codim_K \Si=k$}, if $\codim_K \Si\geq k$ and $\codim_K \Si\not\geq k+1$.
\end{definition}
We use the convention $\codim_K \emptyset=\infty$.  If $Z\inj X$ is a smooth, connected $k$-codimensional submanifold which is not coorientable, then $\codim_\Z Z\geq k+1$.
\begin{definition}\label{def:topvar}
	A connected, closed subset $Z\se X$ is a \emph{topological subvariety of codimension $k$}\index{topological subvariety} if there exists an open subset $U\se Z$ which is a $k$-codimensional submanifold in $X$ and its complement $\Si:=Z\su U$ has $\codim_K \Si\geq k+1$. Such a set $U\se Z$ is called a \emph{fat nonsingular set}.
	
	More generally, a closed subset $Z\se X$ is a \emph{topological subvariety of codimension $k$} if it has finitely many connected components, each of which is a topological subvariety of codimension $k$.
\end{definition}

Topological subvarieties behave similarly to algebraic ones: if $Z\se X$ is a topological subvariety of codimension $k$, then we can define the set of \emph{regular points $Z_{\mathcal{R}}$}, which are the points having neighbourhoods that are locally submanifolds. Then the \emph{singular set $Z_S=Z\su Z_{\mathcal{R}}$} is contained in the complement $\Si$ of any fat nonsingular set, and the long exact sequence of the triple $(X,X\su Z_S,X\su \Si)$ shows that $Z_{\mathcal{R}}$ is also a fat nonsingular set. The main property of a topological subvariety $Z\se X$ relevant to us, is that if $Z$ has a fundamental class, then it is unique up to units of $K$.
\subsubsection{Fundamental class}
Let $Z\se X$ be a topological subvariety of codimension $k$ with fat nonsingular subset $U$, let $y\in U$. A \emph{normal disk $D_y\se X$ to $y\in U$} is a $k$-dimensional smoothly embedded disk $D_y\se X$ centered at $y$, intersecting $U$ in the single point $y$ transversally. The following construction is an extension of the fundamental cohomology class of submanifolds to topological subvarieties.
\begin{definition}
	Let $Z\se X$ be a $k$-codimensional topological subvariety over $K$ and $D_x$ a normal disk of $Z_{\mathcal{R}}$ at $x$. A \emph{local coorientation at $x\in Z_{\mathcal{R}}$}\index{local coorientation} is a generator of $H^k(D_x,D_x\su x)$.
\end{definition}

\begin{definition}\label{def:cycle}
	Let $Z\se X$ be a $k$-codimensional topological subvariety over $K$. A \emph{fundamental cohomology class}\index{fundamental class!topological subvariety} (over $K$) is an element $\braaa X\kettt\in H^k(X,X\su Z;K)$ whose restriction to $H^k(D_x,D_x\su x;K)$ is a generator for all regular points $x\in Z_{\mathcal{R}}$, where $D_x$ is a normal disk over $x$. If such a fundamental class exists, we say that $Z$ is a \emph{cycle}.
\end{definition}
This definition extends the notion of Thom class: if $Z\inj X$ is a submanifold, then it is a cycle iff it is coorientable. Even if a topological subvariety $Z\se X$ is a cycle, the (non-refined) class $[Z]\in H^k(X)$ can be zero (although this cannot happen to the refined class $\braaa Z \kettt\in H^k(X,X\su Z)$). In this section we distinguish the notation of refined class $\braaa Z\kettt$ and $[Z]$, but later we will denote both by $[Z]$. The following proposition summarizes some existence and uniqueness properties of fundamental classes, which follow from the previous discussion (see also \cite[Section A.1.2]{thesis}).
\begin{proposition}\label{prop:cycleexistenceuniqueness}
	Let $Z\se X$ be a $k$-codimensional topological subvariety with a fat nonsingular subset $U$, $\Si=Z\su U$.
	\begin{itemize}
		\item{Uniqueness: Given a fundamental class $\braaa U\kettt\in H^k(X\su \Si,X\su Z)$, $Z$ has at most one fundamental class $\braaa Z\kettt$ restricting to $\braaa U\kettt$.}
		\item{Uniqueness for $K=\F_2$: If $K=\F_2$, then $Z$ has at most one fundamental class $\braaa Z\kettt$.}
		\item{Existence: If $U$ has a fundamental class $\braaa U\kettt$, and if $\codim_K \Si\geq k+2$, then $X$ has a fundamental class restricting to $\braaa U\kettt$.}
	\end{itemize}
\end{proposition}
\begin{remark}\label{rmk:signambiguity}
	If $Z$ is a complex subvariety, then there is a canonical choice for the fundamental class: the complex structure induces an orientation on the normal bundle of the smooth part. If $Z$ is a topological subvariety, and $K=\F_2$, then by the previous proposition, the fundamental class is unique.
	However, for $K=\Z$ and $Z$ connected we have two choices of the fundamental class $[Z]$ according to the choice of the orientation of the normal bundle of $U$.
\end{remark}
\subsection{Excess intersection}\label{sec:eif}
Quillen introduced the excess intersection formula in the context of complex cobordism in \cite{Quillen1971}. We recall Quillen's results, define the excess weight, and then prove the excess weight lemma (Lemma \ref{lemma:mainlemma1}).
		\subsubsection{Clean intersection, excess bundle}
	Smooth submanifolds $Y,Z\inj X$ are said to \emph{intersect cleanly}\index{clean intersection}, if their intersection $W:=Y\cap Z$ is a submanifold and $TY|_W\cap TZ|_W=TW$. The \emph{excess bundle}\index{excess bundle}\label{word:excessbundle} of a clean intersection is $\eta(Y,Z):=TX|_W/(TY|_W+TZ|_W)$. Denoting the inclusion maps
	\begin{displaymath}\tag{EIF}\label{eq:EIFdiagram}
	\xymatrix{
		W \ar@{^{(}->}[r]^{j}\ar@{^{(}->}[d]_{g}& Z\ar@{^{(}->}[d]^f\\
		Y \ar@{^{(}->}[r]^{i} &X
	}\end{displaymath}
	the relations $\nu_i|_W\iso \nu_j\oplus \eta$ and $\nu_f|_W\iso \nu_g\oplus \eta$ hold: for clean intersections the defining short exact sequence of $\eta$ induces
	\begin{displaymath}
	\xymatrix{
		0\ar@{->}[r]^-{}&\mbox{$\underbrace{\left(TY|_W+TZ|_W\right)\big/TZ|_W}_{\nu_g}$}\ar@{->}[r]^-{}&\mbox{$\underbrace{TX|_W\big/TZ|_W}_{\nu_f|_W}$}\ar@{->}[r]^-{}&\eta \ar@{->}[r]^-{}&0
	}\end{displaymath}
	where the first term is isomorphic to $\nu_g$ by the isomorphism theorems. If $f,g$ are cooriented, then there is a unique compatible orientation on $\eta$ such that $\nu_f|_W=\nu_g\oplus \eta$ as oriented bundles.
	
	\begin{remark}\label{rmk:compatibleort}\mbox{}

		\begin{itemize}
			\item The direct sum orientation depends on the order of $\nu_g\oplus \eta$, so let us adopt this convention.
			\item If $f,i,j$ are cooriented, then $\nu_g$ is orientable, with a unique compatible orientation satisfying $$\nu_f|_W\oplus \nu_j=\nu_i|_W\oplus \nu_g$$ as oriented bundles.
		\end{itemize}
	\end{remark}
\subsubsection{Equivariant excess intersection formula}
For the following proposition, see Quillen \cite[Proposition 3.6]{Quillen1971}.
\begin{proposition}\label{prop:eif}
Let $\Ga=\U(1)$ act on $X$. Let $Z\inj X$ be a $\Ga$-invariant oriented smooth submanifold. Then $Z\cap X^\Ga$ is a clean intersection and all maps in
\begin{displaymath}
\xymatrix{
	Z^\Ga \ar@{^{(}->}[r]^{j}\ar@{^{(}->}[d]_{g}& Z\ar@{^{(}->}[d]^f\\
	X^\Ga \ar@{^{(}->}[r]^{i} &X
}\end{displaymath}
can be compatibly oriented, and with these orientations
$$i^*f_!z=g_!\big(j^*z\cdot e(\eta)\big) $$
in $H_\Ga^*(X^\Ga,X^\Ga\su Z^\Ga)$.
\end{proposition}

We will be mainly interested in the special case $z=1$ when we have
\[[Z]|_{X^\Ga}=i^*f_!1=g_!\big( e(\eta)\big). \]

\subsubsection{Excess multiplicity}
Without going into the general theory of real representations we define weights for the three groups we are interested in:
\begin{itemize}
  \item $\Ga=\Z_2$: There are two irreducible real representations, the one dimensional trival, and the one dimensional non-trival one. We define their weights to be 0 and 1 in $\F_2$, respectively.
  \item $\Ga=\U(1)$: There is the one dimensional trivial representation, and for every positive integer $n$ there is a 2-dimensional irreducible real representation. We define their weights to be 0 and $n$ in $\N$, respectively.
  \item $\Ga=\Sp(1)$: We restrict an $\Sp(1)$-representation to its maximal torus, which is a $\U(1)$, and the weights are defined according to the weights of the maximal torus.
\end{itemize}

The multiset of weights is denoted by $W(V)$ for the $\Ga$-representation $V$. So for example $W(V)$ is a set of non-negative numbers with multiplicities for $\Ga=\U(1)$.
The \emph{multiplicity} of $V$ is defined as the product $\mu(V):=\prod_{w\in W(V)}w$.

\begin{remark}\label{euler}
We can always choose an orientation of $V$, such that
\[ e_\Ga(V)=\mu(V)u^k\]
for the $\Ga$-equivariant Euler class of $V$, where $k$ depends on $\Ga$ and the dimension of $V$.
\end{remark}

\begin{definition}\label{def:excessmultiplicity}
Let $(\Ga,R)$ be a halving pair. Let $X$ be a $\Ga$-manifold, $Z\se X$ be a $\Ga$-invariant topological subvariety and  $z\in Z_{\mathcal{R}}^\Ga$. Let $\eta$ denote the excess bundle of the (clean) intersection $Z^\Ga_{\mathcal{R}}=Z_{\mathcal{R}}\cap X^\Ga$. We call the elements of $W(\eta_z)$ the excess weights of $Z$ at $z$. Notice that $W(\eta_z)$ is a subset of the multiset of weights of $\nu_z(X^\Ga\subset X)$. We will call the latter \emph{the normal weights of the halving space $X$ at $z$}. The multiplicity $\mu(\eta_z)$ of the $\Ga$-representation $\eta_z$ is called
 the \emph{excess multiplicity} of $Z$ at $z$.
\end{definition}

\begin{remark}\label{slice} For conjugation spaces and circle spaces, the slice theorem \cite[Theorem B.24]{GuilleminGinzburgKarshon} implies that the excess multiplicity is not zero. For circle spaces it also implies that the rank of the excess bundle $\eta(Z)$ is always even.
\end{remark}

\subsubsection{Equivariant fundamental class}
Let $\Ga$ be a compact connected Lie group and $Z$ be a $\Ga$-invariant topological subvariety with a $\Ga$-invariant fat nonsingular subset
$U$ and singular set $\Si$. For connected groups $\Ga$, existence of a fundamental class is a nonequivariant phenomenon: if $Z$ is a $\Ga$-invariant cycle with fundamental class $[Z]\in H^k(X,X\su Z)$, then there exists a unique $[Z]_\Ga\in H_\Ga^k(X,X\su Z)$ restricting to $[Z]$.

On the other hand, this does not ensure that $Z^\Ga$ is a topological subvariety, since it might happen that $\Si^\Ga$ has too large dimension. This motivates the following definition, introduced for $\Ga=\Z_2$ in \cite{VanHamel2007}:
 \begin{definition}\label{def:goodcycle}
A $\Ga$-invariant closed subset $Z\se X$ is a \emph{good $\Ga$-invariant subvariety of codimension type $(k,l)$} if
   \begin{itemize}
		\item $Z\se X$ is a topological subvariety of codimension $k$ with $\Ga$-invariant fat nonsingular set $U$
		\item $Z^\Ga\se X^\Ga$ is a (nonempty) topological subvariety of codimension $l$ with fat nonsingular set $U^\Ga$.
	\end{itemize}
We call such a set $U$ a \emph{$\Ga$-invariant fat nonsingular set}. If in addition $Z\se X$ is a cycle and its excess multiplicity $\mu_Z:=\mu(\eta_z)$ is independent of $z\in Z_{\mathcal{R}}^\Ga$, then we say that $Z$ is a \emph{good $\Ga$-invariant cycle of codimension type $(k,l)$}\index{good invariant!cycle}.
 \end{definition}

\begin{example}\label{ex:stratifiedgood}
	Let $Z\se X$ be a $k$-codimensional $\Ga$-invariant stratified submanifold with $\Ga$-invariant top stratum $Z_k$. If $Z^\Ga$ has a stratification whose \emph{unique top stratum is $(Z_k)^\Ga$}, then $Z$ is a good $\Ga$-invariant subvariety of codimension type $(k,l)$ where $l=\codim (Z_k)^\Ga$. If additionally $Z$ is a $\Ga$-invariant cycle and $(Z_k)^\Ga$ is connected, then $Z$ is a good $\Ga$-invariant cycle. This will be relevant in the case of real double Schubert varieties, see Theorem \ref{thm:doubleflagcirclespace}.
\end{example}

For good $\Ga$-invariant cycles, the classes $[Z\se X]_\Ga$ and $[Z^\Ga\se X^\Ga]$ are related by the following Lemma.

\begin{lemma}[Excess multiplicity lemma]\label{lemma:mainlemma1}
	Let $\Ga=\U(1)$ and $R=\Q$ be the coefficients of cohomology. Let $X$ be a $\Ga$-manifold and let $Z\se X$ be a good $\Ga$-invariant cycle of codimension type $(k,l)$. Then $Z^\Ga\se X^\Ga$ is a cycle and it has a fundamental class $[Z^\Ga\se X^\Ga]$ satisfying
	\begin{displaymath}
	\xymatrix@R-2pc{
		H_\Ga^*(X) \ar@{->}[r]^-{r}& H_\Ga^*(X^\Ga)\iso H^*(X^\Ga)[u]\otimes H_\Ga^*\\
		[Z\se X]_\Ga\ar@{|->}[r]^{} &w \cdot [Z^\Ga\se X^\Ga]+\deg_\Ga^{<k-l}
	}\end{displaymath}
	where $r=i^*$ is the restriction to $X^\Ga$, $w=e_\Ga(\eta_z)\in H^{k-l}_\Ga$ is the equivariant Euler class of the excess bundle $\eta$ at $z$ for some $z\in Z_{\mathcal{R}}^\Ga$, (see Remark \ref{euler}) and $\deg_\Ga^{<k-l}$ denotes a sum of homogeneous elements whose $H_\Ga^*$-degree is less than $k-l$.
\end{lemma}	
\begin{proof}
	First, consider the case when $Z\se X$ is a smooth $\Ga$-invariant cycle. Then there is a commutative diagram of equivariant inclusions:
	\begin{displaymath}
	\xymatrix{
		Z^\Ga \ar@{^{(}->}[r]^{j}\ar@{^{(}->}[d]_{g}& Z\ar@{^{(}->}[d]^f\\
		X^\Ga\ar@{^{(}->}[r]^{i} &X
	}\end{displaymath}
	By Proposition \ref{prop:eif}, $Z^\Ga=Z\cap X^\Ga$ is a clean intersection and therefore has an excess bundle $\eta$. Since $X^\Ga\inj X$ and $Z^\Ga\inj Z$ are cooriented (their normal bundle is complex via the $\Ga$-action), and by Remark \ref{rmk:compatibleort}, this induces a compatible orientation of $g$ and therefore on $\eta$. The equivariant Euler class of any equivariant bundle $E\to X$ of real rank $n$ over $X$ connected and fixed by $\Ga$ can be written as $e_\Ga(E)=e_\Ga(E_x)+\deg_\Ga^{<n}$. Applying this to $\eta\to Z^\Ga$,
	$$ e_\Ga(\eta)=w +\deg_\Ga^{<k-l}$$
	and by the equivariant EIF and using that $g_!$ is a $H_\Ga^*$-module homomorphism
	$$i^*f_!1=g_!(e_\Ga(\eta))=w\cdot g_!1+\deg_\Ga^{<k-l},$$
	proving the claim for $Z$ smooth.\\
	
	Now let $Z=U\amalg \Sigma$ be a good $\Ga$-invariant cycle of $X$ of codimension $(k,l)$ with $\Ga$-invariant fat nonsingular set $U$. Consider the following commutative diagram:
	\begin{displaymath}
	\xymatrix{
		H^*_\Ga(X,X\su Z) \ar@{->}[r]^{s}\ar@{->}[d]_{r}& H^*_\Ga(X\su \Sigma,X \su Z)\ar@{->}[d]^{r_U}\\
		H^*(X^\Ga,X^\Ga\su Z^\Ga)[u] \ar@{->}[r]^-{s_U} &H^*(X^\Ga\su \Sigma^\Ga,X^\Ga \su Z^\Ga)[u]
	}\end{displaymath}
Since $Z^\Ga\se X^\Ga$ is an $l$-codimensional singular subvariety, by definition, $H^i(X^\Ga,X^\Ga\su Z^\Ga)=0$ for $i<l$. So one has
	$$ r[Z]_\Ga=\sum_{i=0}^{(k-l)/2} \xi_{k-2i} u^i$$
for $\xi_i\in H^{2i}(X^\Ga,X^\Ga\su Z^\Ga)$, where $k-l$ is even. By commutativity of the diagram,
	\begin{equation}\label{eq:restriction}
	\sum_{i=0}^{(k-l)/2} s_U(\xi_{k-2i}) u^i= s_U r[Z]_\Ga=r_U s[Z]_\Ga=r_U[U]_\Ga = w \cdot [U^\Ga]+\deg_\Ga^{<k-l}
	\end{equation}
	
for a fundamental class $[U^\Ga]$, by the first case, since $U\inj X\su \Si$ is smooth. Therefore $s_U(\xi_l)u^{(k-l)/2} =w\cdot [U^\Ga]$. Write $w=\mu u^{(k-l)/2}$, then $[Z^\Ga]:=\xi_l/\mu$ is a fundamental class with the required property ($\mu\neq 0$ by Remark \ref{slice}).
\end{proof}
\begin{remark}\label{rmk:excessweightlemma}\mbox{}
	\begin{itemize}
		\item[i)] If $R=\Z$, then the proof also shows that equation \eqref{eq:restriction} holds, but then $\xi_l$ might not be divisible by $\mu$, and therefore $Z^\Ga\se X^\Ga$ is not necessarily a cycle.
		\item[ii)] Let $Z\se X$ be a good $\Ga$-invariant subvariety and a cycle. If $W(\nu(X^\Ga\se X))$ consists of a single weight $\la$, then the excess multiplicity is $\la^{d/2}$, independently of $z\in Z_{\mathcal{R}}^\Ga$. Therefore $Z$ is a \emph{good} $\Ga$-invariant cycle.
		\item[iii)] In the proof we used that $0\not\in W(\nu(X^\Ga\inj X))$, which always holds for $\Ga=\U(1)$ as indicated in Remark \ref{slice}.
		For $\Ga=\Sp(1)$, this is not necessarily true; for example using the isomorphism $\R^4\cong\HH$ we get an $\Sp(1)$-action on $ \Gr_{4k}(\R^{4n})$ with fixed point set $\Gr_k(\HH^n)$ whose excess multiplicity is 0. If one assumes $\mu_Z\neq 0$, and that $Z^\Ga\se X^\Ga$ is a cycle, the Lemma also holds for $\Ga=\Sp(1)$, $D=4$.
	\end{itemize}
\end{remark}

\subsection{Halving cycles}

Adapting the proof of van Hamel \cite[Cor 1.3(ii)]{VanHamel2007} to the context of circle spaces shows that $\ka[Z]=\mu[Z^\Ga]$, where $\mu$ is an undetermined constant. This constant is explicitly described as the excess multiplicity as described in Section \ref{sec:eif}, and allows us to generalize the Borel-Haefliger theorem to $\U(1)$ and $\Sp(1)$-actions.

\begin{definition}\label{def:halvingcycle}\index{halving cycle}
	Let $X$ be a $\Ga$-manifold. A good $\Ga$-invariant cycle $Z\se X$ (Definition \ref{def:goodcycle}) of codimension type $(2k,k)$ is called a \emph{halving cycle}.
\end{definition}
For example, in the algebraic case, complexified cycles $Z_\C$, are $\Z_2$-halving cycles over $R=\F_2$.
\begin{remark}\label{rmk:halvingdim4}
	For $\Ga=\U(1)$, a halving cycle $Z$ has codimension divisible by 4. Indeed, the codimension of $Z\se X$ has the same parity as the codimension of $Z^\Ga\se X^\Ga$, by Remark \ref{slice}).
\end{remark}

Now we are ready to generalize Theorem \ref{pr:conj-ew} of Borel and Haefliger to halving manifolds:
\begin{theorem}\label{prop:halvingcycle}
	Let $(\Ga,R)$ be $(\Z_2,\F_2)$ or $(\U(1),\Q)$. Let $X$ be a halving manifold, and $Z\se X$ be a halving cycle. Then
	$$ \si[Z]=[Z]_\Ga,\qquad \ka[Z]=\mu_Z[Z^\Ga],$$
	where $\mu_Z$ is the excess multiplicity of $Z$ (Definition \ref{def:excessmultiplicity}).
\end{theorem}
\begin{proof}
	Set $\codim Z=2Dk$. By the excess multiplicity lemma (Lemma \ref{lemma:mainlemma1})
	\begin{equation}\label{eq:excessrestriction}
	r[Z]_\Ga=w\cdot [Z^\Ga]+\deg_\Ga^{<Dk},
	\end{equation}
	where $\deg_\Ga^{<Dk}$ denotes a polynomial in $u$ of degree less than $k$ and $w=e_\Ga(\eta_z)$.	By the degree lemma (Lemma \ref{lemma:degreelemma}) $[Z]_\Ga=\si(x)$ for some $x\in H^{2Dk}(X)$. Since $$[Z]=\rho[Z]_\Ga=\rho\si(x)=x,$$ $[Z]_\Ga=\si(x)=\si[Z]$. Restricting, using the definition of $\ka$ \eqref{eq:kadef} and \eqref{eq:excessrestriction},
	$$ \ka[Z]u^{k}+\deg_\Ga^{<Dk}=r\si[Z]=r[Z]_\Ga=w\cdot [Z^\Ga] +\deg_\Ga^{<Dk}$$
	therefore $\ka[Z]=\mu_Z\cdot [Z^\Ga]$, where $w=\mu_Z u^k$.
\end{proof}
\begin{remark}\label{rmk:Sp1halvingcycle}
	The lemma also holds for the halving pair $(\Ga,R)=(\Sp(1),\Q)$, if one assumes that $Z^\Ga\se X^\Ga$ is a cycle and that $w\neq 0$, see Remark \ref{rmk:excessweightlemma} iii).
\end{remark}

\subsubsection{Coefficients of the restriction equation}
Let us return for a moment to the case of conjugation spaces; let $X$ be a conjugation space. Franz and Puppe \cite{FranzPuppe2006} determined the coefficients of the restriction equation as Steenrod squares:
\begin{equation}\label{eq:FranzPuppe}
 r(\si(\al)) = \sum_{i=0}^d \Sq^i \al\cdot u^{d-i}=:\Sq^u(\al),
\end{equation}
for $\al\in H^{2d}(X)$. Together with the theorem of Van Hamel \cite{VanHamel2007}, this proves the topological version of a classical theorem of Chow \cite{Chow1963}, that $[Z]|_{X^\Ga}=[Z^\Ga]^2$. Using Proposition \ref{prop:halvingcycle}, we obtain a simple proof of a weaker version of \eqref{eq:FranzPuppe}, namely in the algebraic case. For a similar theorem, see \cite[Theorem 1.18]{BenoistWittenberg2018}. 
\begin{proposition}\label{prop:Steenrod}
	Let $X$ be {the complexification of a real algebraic variety which is smooth, and let $Z\se X$ be a complexified subvariety which is a smooth cycle}. Then
	$$r\si [Z]=\Sq^u(\ka[Z])$$
\end{proposition}
\begin{proof}
	The proof can be summarized as
	$$ [Z]_\Ga|_{X^\Ga}\operatornamewithlimits{=\joinrel=}^{1)}g_!(e_\Ga(\eta))\operatornamewithlimits{=\joinrel=}^{2)}g_!(w_*^u(\nu))\operatornamewithlimits{=\joinrel=}^{3)}\Sq[Z^\Ga],$$
	where $g:Z^\Ga\inj X^\Ga$ and $w_*^u$ denotes the total homogenized Stiefel-Whitney class. 1) is the excess intersection formula. Since $Z$ is a complexification,
	$$\nu(Z\se X)|_{Z^\Ga}=\nu(Z^\Ga\se X^\Ga)\otimes_\R \C$$
	holds. Then the excess bundle equivariantly is $\eta=\nu(Z^\Ga\se X^\Ga)\otimes_\R i\R$ where $i\R$ denotes the trivial line bundle with nontrivial $\Z_2$-action. This implies 2): $e_\Ga(\nu\otimes_\R i\R)=w_*^u(\nu)$ (the total Stiefel-Whitney class homogenized by powers of $u$). Finally, 3) is the content of Thom's theorem \cite{Thom1950} saying that for $i:Z\inj X$: $i_!(w_*(\nu))=\Sq[Z]$ holds.
\end{proof}

It would be nice to have a similar description of the coefficients for the case of circle spaces, however there are no nontrivial stable rational cohomology operations, so this result has no direct generalization.
\subsubsection{Poincar\'e duality}
In this section we give some sufficient conditions for a $\Ga$-space $X$ to be a halving space.

\begin{definition}\label{def:almosthalvingspace}\index{halving space!almost halving space}
	We say that a $\Ga$-space $X$ is \emph{almost a halving space}, if $X$ has nonzero cohomology only in degrees $2Di$ and $X$ is equivariantly formal with a Leray-Hirsch section $\si:H^*(X)\to H_\Ga^*(X)$ satisfying a weaker form of the degree condition:
	\begin{itemize}
		\item[(DC\textsuperscript{--})] for all $x\in H^{2Di}(X)$, $r\si(x)$ is a polynomial of degree \emph{at most} $i$ where $r:H_\Ga^*(X)\to H^*(X^\Ga)[u]$ is the restriction map, $u\in H_\Ga^D$.
	\end{itemize}
\end{definition}
To put it simply, (DC\textsuperscript{--}) allows the $u$-degree of $r\si(x)$ to be smaller than $i$. The following lemma can be found (implicitly) in van Hamel \cite{VanHamel2007} for the case of conjugation spaces and its proof is the same. For the analogue of van Hamel's theorem in the case of $(\Ga,R)=(\U(1),\Q)$ we also have to assume Poincar\'e duality/orientability:

\begin{lemma}[Injectivity lemma]\label{lemma:injectivity}
	Let $(\Ga,R)=(\U(1),\Q)$, $H_\Ga^*\iso \Q[u]$, $u\in H_\Ga^D$, $D=2$. Let $X$ be a smooth $\Ga$-manifold which is almost a halving space with $\si$. If $X$ is compact, orientable and $\dim X\geq 2\dim X^\Ga$, then $X$ satisfies the degree condition \emph{(DC)}. In particular, $X$ is a halving space with the same $\si$, and $\ka$ defined by \eqref{eq:kadef} is an isomorphism.
\end{lemma}
\begin{proof}We only give a brief sketch of the main idea found in \cite{VanHamel2007}, for further details see \cite{thesis}. Let $x\in H^{2Dk}(X)$. By Poincar\'e duality there exists $y\in H^{2D(n-k)}(X)$ satisfying $xy\neq0$, here $\dim X= 2Dn$. Then $\si(x)\si(y)\neq 0$; degree considerations from (DC\textsuperscript{--}) show that $r(\si(x)\si(y))=\ka(x)\ka(y)u^n$. The localization theorem (restriction to the fixed-point set $r:H_\Ga^*(X)\to H^*(X^\Ga)[u]$ is an isomorphism after inverting $u$) implies that restriction to the fixed point set $r:H^*_\Ga(X)\to H^*(X^\Ga)[u]$ is injective, therefore $\ka(x)\neq0$.
\end{proof}
\begin{remark}\label{rmk:injectivitylemma} \mbox{}

	\begin{itemize}
		\item[i)] More generally, the lemma holds for \emph{$\Q$-Poincar\'e duality spaces} \cite[Definition 5.1.1.]{AlldayPuppe1993} if one replaces $\dim(X)$ with formal dimension $\fd(X)$. ($X$ is a $\Q$-Poincar\'e duality space if $H^{\operatorname{top}}(X;\Q)\iso \Q$ and the pairing $$H^k(X)\otimes H^{\operatorname{top}-k}(X)\to H^{\operatorname{top}}(X)$$ is perfect; $\fd(X):=\operatorname{top}$.) $\Q$-Poincar\'e duality is satisfied by a larger class of spaces, which need not be orientable nor compact. {Note that the formal dimension of a manifold $X$ can be smaller than the dimension of $X$ as a manifold. For example, $\RP^{2n}$ is a $\Q$-Poincar\'e duality space with formal dimension 0. More generally, all real partial flag manifolds $\Fl_\D^\R$ are $\Q$-Poincar\'e duality spaces \cite{realflagpaper}.} This allows us to extend the class of circle space examples to the case of nonorientable Grassmannians, $\Gr_k(\R^n)$ for $n$ odd. For the sake of conciseness we didn't add the extra conditions to the lemma.
		\item[ii)] For $(\Ga,R)=(\Z_2,\F_2)$, the same lemma holds, without having to assume orientability. Every manifold satisfies $\F_2$-Poincar\'e duality -- indeed, van Hamel's original proof \cite{VanHamel2007} does not assume orientability.
		\item[iii)] The lemma can also be generalized to $\Ga=\Sp(1)$, if one makes the additional assumption that the restriction to the fixed point set is an isomorphism after localization. 	
	\end{itemize}
\end{remark}

\section{The generalized Borel-Haefliger theorem}\label{sec:BHI}
\noindent Using the notions of good invariant cycles and excess multiplicity (Definitions \ref{def:goodcycle} and \ref{def:excessmultiplicity}) we can state the generalized Borel-Haefliger theorem:
\begin{theorem}[Generalized Borel-Haefliger theorem]\label{thm:BHI}\index{Borel-Haefliger theorem}
	Let $(\Ga,R)$ be $(\Z_2,\F_2)$ or $(\U(1),\Q)$ and $H^*(B\Ga;R)\iso R[u]$, $u\in H_\Ga^D$. Let $X$ be a smooth compact $\Ga$-manifold, orientable if $R=\Q$. Assume that $H^*(X)$ is nonzero only in degrees divisible by $D$ and has a basis of halving cycles: good $\Ga$-invariant cycles $Z_i\se X$ of codimension type $(2Dk_i,Dk_i)$. Then $X$ is a halving space with cohomology frame
	\begin{equation} \label{eq:BHI}
	\si[Z_i]=[Z_i]_\Ga,\qquad \ka[Z_i]=\mu_i [Z_i^\Ga]
	\end{equation}	
	where $\mu_i\neq 0$ is the excess multiplicity of $Z_i\se X$.
\end{theorem}
\begin{proof}
	Since $H^*(X)$ is generated by halving cycles, it is $\Ga$-equivariantly formal. By the excess weight lemma (Lemma \ref{lemma:mainlemma1}),
	$$r\si[Z_i]=\ka[Z_i] u^{k_i}+\deg_\Ga^{<Dk_i},\qquad \text{for all }i,$$ where $\deg_\Ga^{<Dk_i}$ denotes a polynomial in $u$ of degree less than $k_i$.  Since $[Z_i]$ form a basis of $H^*(X)$, $X$ is almost a halving space, cf.\ Definition \ref{def:almosthalvingspace}. Since $X$ is a compact orientable manifold, the class of a point is a cycle which is represented by a halving cycle by assumption, so $\dim X=2\dim X^\Ga$. Therefore $X$ is a halving space with cohomology frame $(\ka,\si)$ by the injectivity lemma (Lemma \ref{lemma:injectivity}).
\end{proof}
\begin{remark}\label{rmk:BHI}\mbox{}

	\begin{itemize}
		\item[i)] The orientability assumption is not essential, see Remark \ref{rmk:injectivitylemma}. However, then one has to add conditions on the formal dimensions: $\fd(X)\geq 2\fd(X^\Ga)$, for the injectivity lemma to hold.
		\item[ii)]	For $(\Ga,R)=(\Sp(1),\Q)$ there is an analogous theorem, if one makes additional assumptions: first, that the $Z_i^\Ga$ are cycles, and second, that the localization theorem holds (the injectivity lemma for $(\Sp(1),\Q)$ only holds under this assumption, see Remark \ref{rmk:injectivitylemma}, iii)).
	\end{itemize}
\end{remark}

\section{Examples}\label{sec:examples}
Since conjugation spaces are discussed in \cite{HausmannHolmPuppe2005}, we concentrate on the case of circle spaces.
\subsection{$R$-spaces}\label{sec:circleexamples}
Our main class of examples of circle spaces are homogeneous spaces, in particular they are all \emph{$R$-spaces}, more commonly known as (generalized) real flag manifolds. We reserve the terminology real flag manifolds to $\Fl_\D(\R^N)$, which are also $R$-spaces.

Out of the $R$-spaces, we have four main classes of examples of circle spaces: spheres $S^{4n}$, even real flag manifolds, quaternionic and octonionic flag manifolds. The simplest example consists of spheres $S^{4n}$; this already illustrates the idea of the proof. Let $\U(1)$ act on $\R^{4n+1}$ as the linear orthogonal representation, which splits into $n$ weight one and $2n+1$ trivial representations; $S^{4n}\se \R^{4n+1}$ is $\U(1)$-invariant.
\begin{proposition}\label{prop:s4n}
	With this $\Ga$-action $S^{4n}$ is a circle space with fixed point set $S^{2n}$.
\end{proposition}
\begin{proof}
	The fixed point set of $S^{4n}$ is $S^{4n}\cap \R^{2n+1}=S^{2n}$. Since $H^*(S^{4n})=\Z[x]/(x^2)$ is generated by a $\Ga$-invariant halving cycle, namely the class of a fixed point, by the generalized Borel-Haefliger theorem, $S^{4n}$ is a circle space.
\end{proof}

\subsubsection{Real flag manifolds}\label{subsec:realflags}

Our first nontrivial class of examples are the \emph{even real flag manifolds}\label{word:evenrealflag}, i.e.\ flag manifolds $\Fl_{2\D}(\R^{2n})$, where all dimensions are even. More details about this example can be found in \cite{realflagpaper}, we give a brief summary. For $\D=(d_1\stb d_m)$, we denote by $\Fl_\D(\R^n)$ the manifold of flags $F_\bullet=(F_1\se \ldots F_m)$, whose dimensions are $\dim F_i=s_i$, where $s_i=\sum_{j=1}^id_j$.

The identification of $\R^{2n}\leftrightarrow \C^n$ as real $\Ga$-representations induces an action on $\Fl_{\mathcal{E}}(\R^{2n})$, $\mathcal{E}=(e_1\stb e_r)$. The flag manifold $\Fl_{\mathcal{E}}(\R^{N})$ has a Schubert cell decomposition
\begin{equation}\label{eq:Schubertincidence0} \Om_I(A_\bullet) = \{ F_\bullet\in \Fl_{\mathcal{E}}(\R^N): \dim F_i\cap A_{k}=r_I(i,k)\},\end{equation}
where $$I\in \OSP(\mathcal{E})= S_N/(S_{e_1}\times \ldots \times S_{e_r})$$
is an \emph{ordered set partition}\label{word:OSP}, and
$r_I(i,k)=\#\{l\in I_1\cup \ldots \cup I_i: l\leq k\}$. If $2\D=(2d_1,2d_2\stb 2d_r)$ and $I\in \OSP(\D)$, then \emph{the doubled ordered set partition} $DI\in \OSP(2\D)$ is obtained by replacing each $i\in I_j$ by $(2i-1,2i)\in DI_{j}$. A \emph{double Schubert variety}\label{word:doubleSchubert} $\si_{DI}^\R\se \Fl_{2\D}^\R$ is a Schubert variety corresponding to $DI\in \OSP(2\D)$. In the case of the Grassmannian $\D=(k,l)$, $DI\in \binom{2(k+l)}{2k}$ corresponds to the Young diagram obtained by subdividing each square into $2\times 2$ squares in the Young diagram corresponding to $I\in \binom{k+l}{k}$, see Figure \ref{fig:doubleYoung}.

\begin{figure}
	\centering
	\begin{picture}(50,30)
	\put(0,30){\line(0,-1){30}}
	\put(10,30){\line(0,-1){30}}
	\put(20,30){\line(0,-1){20}}
	\put(30,30){\line(0,-1){10}}	
	\put(40,30){\line(0,-1){10}}	
	\put(0,30){\line(1,0){40}}
	\put(0,20){\line(1,0){40}}
	\put(0,10){\line(1,0){20}}
	\put(0,0){\line(1,0){10}}
	\put(50,15){\vector(1,0){10}}
	\end{picture}$\quad$
	\begin{picture}(50,30)
	\linethickness{0.1mm}
	\put(5,30){\line(0,-1){30}}
	\put(15,30){\line(0,-1){20}}
	\put(25,30){\line(0,-1){10}}	
	\put(35,30){\line(0,-1){10}}		
	\put(0,25){\line(1,0){40}}			
	\put(0,15){\line(1,0){20}}				
	\put(0,5){\line(1,0){10}}					
	\linethickness{0.4mm}
	\put(0,30){\line(0,-1){30}}
	\put(10,30){\line(0,-1){30}}
	\put(20,30){\line(0,-1){20}}
	\put(30,30){\line(0,-1){10}}	
	\put(40,30){\line(0,-1){10}}	
	\put(0,30){\line(1,0){40}}
	\put(0,20){\line(1,0){40}}
	\put(0,10){\line(1,0){20}}
	\put(0,0){\line(1,0){10}}
	
	\end{picture}
	\caption{The double of a Young diagram}
	\label{fig:doubleYoung}
\end{figure}
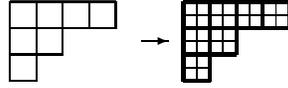

The double Schubert varieties $\si_{DI}^\R$ are cycles and their classes form a basis of $H^*(\Fl_{2\D}^\R;\Q)$, see \cite{realflagpaper}. Using that they are circle spaces we can deduce their structure constants.
\begin{theorem}\label{thm:doubleflagcirclespace}
	Let $\Ga$ act on $\Fl_{2\D}(\R^{2n})$ obtained by identifying $\C^n$ with $\R^{2n}$. With this action it is a circle space, with fixed point set $\Fl_\D(\C^n)$. Furthermore $$\ka[\si_{DI}^\R]=2^{|I|}[\si_I^\C],$$ where $[\si_I^\C]\in H^{2|I|}(\Fl_\D(\C^N))$.
\end{theorem}
\begin{proof}
	Let $F_\bullet$ be a complete flag in $\R^{2n}$, such that $F_{2i}$ are $\Ga$-invariant, and let $F_\bullet^\C$ denote the corresponding complex flag in $\C^n$. By the generalized Borel-Haefliger theorem, it is enough to show that the Schubert varieties $\si_{DI}^\R(F_\bullet)$ are halving cycles, have fixed point set $\si_I^\C(F_\bullet^\C)$ and that $[\si_{DI}^\R]$ form a basis.  For further details see \cite{realflagpaper}.
	
	We sketch why all normal weights of $\Fl_\D(\C^n)\inj \Fl_\D(\R^{2n})$ are 2. Since all tangent spaces are sums of $\Hom$-spaces, this claim can be reduced to linear algebra, namely computing weights of the $\U(1)$-representation $\Hom_\R(\C,\C)$. This representation splits into the sum of two $\U(1)$-representations: a 2-dimensional weight 0 representation and a weight 2 representation. The weight 0 representation corresponds to $\Hom_\C(\C,\C)$ (in the geometric picture this correponds to the tangent space of the complex part) and the weight 2 representation $\Hom_{\conj{\C}}(\C,\C)$ (the normal space of the complex part).
\end{proof}
The corollaries below follow from Theorem \ref{thm:doubleflagcirclespace}, multiplicativity of $\ka$ and that
	$$\ka p_j(S^\R_i)=2^jc_j(S^\C_i).$$
\begin{corollary}[Littlewood-Richardson coefficients]\label{cor:realLR}
	$$[\si_{DI}^\R]\cdot[\si_{DJ}^\R]=\sum_{K}c_{IJ}^K [\si_{DK}^\R]$$
	where $c_{IJ}^K$ are the complex Littlewood-Richardson coefficients
	$$[\si_{I}^\C]\cdot[\si_{J}^\C]=\sum_{K}c_{IJ}^K [\si_{K}^\C].$$
\end{corollary}

\begin{corollary}[Giambelli formula type description]
	\label{cor:realBGG}
	$$[\si_{DI}^\R]=q(p_*(S_i^\R))\qquad \iff \qquad [\si_{I}^\C]=q(c_*(S_i^\C)),$$
	i. e. the same polynomial describes the double real Schubert classes and complex Schubert classes in terms of Pontryagin and Chern classes.
\end{corollary}
\begin{corollary}\label{cor:realrelations} The cohomology ring of an even flag manifold can be described as follows:
	$$ H^*(\Fl_{2\D}^\R)=\Q[p_*(S_i^\R)]/\mathcal{R}(p_*(S_i^\R))\acsa H^*(\Fl_\D^\C)=\Q[c_*(S_i^\C)]/\mathcal{R}(c_*(S_i^\C)),$$
	where $\mathcal{R}(x_*^i)$ denotes an ideal in the variables $x_j^i$, that is the same polynomial relations hold in the two cohomology rings in terms of Pontryagin and Chern classes of the respective tautological bundles.
\end{corollary}
\begin{corollary}[Equivariant Giambelli formula]
	For the case of Grassmannians $\Fl_{2\D}^\R=\Gr_{2k}(\R^{2(k+l)})$, $\D=(k,l)$, the (doubled) Giambelli formula holds, even $\Ga$-equivariantly
	$$ 	[\si_{D\la}]_\Ga=\det
	\left|
	\begin{array}{cccc}
	[\si_{D\la_1}]_\Ga  & [\si_{D(\la_1+1)}]_\Ga  &   \ldots &  [\si_{D(\la_1+k)}]_\Ga \\
	
	[\si_{D(\la_2-1)}]_\Ga  &[\si_{D\la_2}]_\Ga  &\ldots   &[\si_{D(\la_2+k-1)}]_\Ga \\
	\vdots  & \vdots  &\ddots   &\vdots\\
	
	[\si_{D(\la_k-k)}]_\Ga  &  \ldots & \ldots  &[\si_{D\la_k}]_\Ga
	\end{array}
	\right|$$
	where $D\la$ denotes the double of the partition $\la\se k\times l$ and $Da=(2a,2a)$ for $a\in \Z$.
\end{corollary}
\begin{proof}
	Nonequivariantly, this follows from the complex Giambelli formula and from $X$ being a circle space with $\ka[\si_{D\la}^\R]=2^{|\la|}[\si_\la^\C]$. Equivariantly, this follows from $\si$ being multiplicative (Corollary \ref{thm:multiplicativity}) and from the generalized Borel-Haefliger theorem, $\si[\si_{D\la}]=[\si_{D\la}]_\Ga$.
\end{proof}

The Grassmannians $\Gr_{K}(\R^N)$ are also circle spaces, except when $K$ is odd, $N$ is even. If $K$ and $N$ is even, this is contained in Theorem \ref{thm:doubleflagcirclespace}. The remaining cases: $K$ odd $N$ even and $K$ even $N$ odd are both nonorientable, in these cases Remark \ref{rmk:injectivitylemma}, i) can be used. This gives examples of nonorientable circle spaces. For further details, see \cite{realflagpaper}.

\subsubsection{Galois type actions}\label{subsec:Galois}
For the next examples of halving spaces, we define $\U(1)$-actions on $\HP^n$ with fixed point set $\CP^n$ and $\Sp(1)$-actions on $\OP^2$ with fixed point set $\HP^2$. As a first step, let us recall the corresponding actions on the real normed division algebras. There are four real normed division algebras $\F_i$:
$$ \R\se \C \se \HH \se \O.$$
In each case, there is a subgroup $\Ga\iso \operatorname{O}(\F_{i-1})$ of the $\R$-algebra automorphisms $\Aut(\F_i)$, whose fixed point set $\F_i^\Ga$ is the previous division algebra $\F_{i-1}$, $i=2,3,4$ (hence the name ``Galois type''). We briefly recall how these actions arise.

First, the group of (continuous!) automorphisms is well-known to be $\Z_2$ in the case of $\C$, with fixed point set $\R$. This action extends to $\C^n$ and also to the complex flag manifolds $\Fl_\D(\C^N)$ with fixed point-set $\Fl_\D(\R^n)$. This is the action classically studied by Borel and Haefliger.

Next, let $\HH$ act on itself by inner automorphisms. Then $\Ga=\U(1)\se \C=\bra 1,i\ket\leq \HH$ acts on $\HH$, with fixed-point set $\C\leq \HH$. This action extends to $\HH^n$ therefore to $\HP^n$, and even to any quaternionic flag manifold $\Fl_\D(\HH^n)$, with fixed point set $\Fl_\D(\C^n)$.

The automorphisms of the normed algebra $\O$ fixing $\HH$ is isomorphic to $\Ga:=\Sp(1)$, and in fact $\O^\Ga=\HH$. This induces an action on the octonionic flag manifolds, which can be seen on their different models -- we describe these actions in Section \ref{subsec:octoflag}. For additional details, see e.g. \cite[Propositions B.4.1, B.4.2]{thesis}.
\subsubsection{Quaternionic flag manifolds}
The flag manifold $\Fl_\D(\HH^n)$ has a Schubert cell decomposition $\si_I^\HH$ where $I\in \OSP(\D)$.
\begin{theorem}\label{thm:quatflagcirclespace}
	With the $\Ga=\U(1)$-action defined by inner automorphisms (see Section \ref{subsec:Galois}), $\Fl_\D(\HH^n)$ is a circle space, with fixed point set $\Fl_\D(\C^n)$. Furthermore
	$$ \ka[\si_I^\HH]=2^{|I|}[\si_I^\C],$$
	where $[\si_I^\C]\in H^{2|I|}(\Fl_\D(\C^n))$.
\end{theorem}
\begin{proof}
	If $F_\bullet^\C$ is a complex flag, then $F_\bullet:=F_\bullet^\C\otimes_\C\HH$ is a quaternionic flag, which is $\Ga$-invariant. Similarly to the case of real even flag manifolds, by the generalized Borel-Haefliger theorem it can be shown that the Schubert varieties $\si_I^\HH(F_\bullet)$ are halving cycles with respect to an appropriate complete flag $F_\bullet\in \Fl(\HH^n)$, with fixed points $\si_I^\C(F_\bullet^\C)$ and $[\si_I^\HH]$ form a basis of rational cohomology. The normal weights are all $2$.
\end{proof}
\begin{corollary}[Littlewood-Richardson coefficients]
	\label{cor:quaternionicLR}
	$$[\si_{I}^\HH]\cdot[\si_{J}^\HH]=\sum_{K}c_{IJ}^K [\si_{K}^\HH]$$
	where $c_{IJ}^K$ are the same Littlewood-Richardson coefficients as in
	$$[\si_{I}^\C]\cdot[\si_{J}^\C]=\sum_{K}c_{IJ}^K [\si_{K}^\C].$$
\end{corollary}
\begin{proof}
	Exactly the same as Corollary \ref{cor:realLR}.
\end{proof}
\begin{corollary}[Giambelli formula type description]
	$$[\si_{I}^\HH]=q(p_*(S_i^\HH))\qquad \iff \qquad [\si_{I}^\C]=q(c_*(S_i^\C)),$$
	where $p_*$ denotes quaternionic Pontryagin classes. In words, the same polynomial describes the quaternionic and complex Schubert varieties in terms of characteristic classes.
\end{corollary}
In the case of Grassmannians, this was already noticed by Pragacz and Ratajski \cite{PragaczRatajski1997}; as they remark, the proof of the Pieri formula in \cite{GriffithsHarris1978} can be replicated in the quaternionic case implying the same description of the cohomology rings (complex and quaternionic) with degrees doubled.

\subsubsection{Octonionic flag manifolds}\label{subsec:octoflag}
In this section we give some examples for quaternionic halving spaces, i.e.\ halving spaces for $\Sp(1)$-actions: the octonionic flag manifolds. We will also show that they are circle spaces by restricting to $\U(1)\leq \Sp(1)$.

By octonionic flag manifolds we mean the following three examples: $\OP^1, \OP^2, \Fl(\O)(=\Fl(\O^3))$. Nonassociativity of octonions leads to the fact that there are no octonionic analogues of higher dimensional flag manifolds. We refer to \cite{Baez2002}, \cite{Freudenthal1985}, \cite{Eschenburg}, \cite{MareWillems2013} for further details about octonionic flag manifolds.

Since $\OP^1\iso S^8$ which is easily seen to be both a circle space and a quaternionic halving space, we start with $\OP^2$. In the case of $\OP^2$ a purely topological proof can be given using Hopf fibrations.
\begin{proposition}
	The Hopf fibrations are $\Ga$-equivariant principal $G$-bundles where the $\Ga$-action is induced by the inner automorphisms defined in Section \ref{subsec:Galois}. Furthermore, the $\Ga$-fixed point set  of each Hopf fibration is the previous one:
	\begin{itemize}
		\item $\pi_2:S^1\to S^3\to S^2$ is $\Ga=\Z_2$-equivariant with fixed point set $\pi_1:S^0\to S^1\to S^1$
		\item $\pi_3:S^3\to S^7\to S^4$ is $\Ga=\U(1)$-equivariant with fixed point set $\pi_2:S^1\to S^3\to S^2$
		\item $\pi_4:S^7\to S^{15}\to S^8$ is $\Ga=\Sp(1)$-equivariant with fixed point set $\pi_3:S^3\to S^7\to S^4$		
	\end{itemize}
\end{proposition}
\begin{proof}
	The definition of these bundles involves the division algebra structure of $\F$, so they are naturally $\Aut(\F)$-equivariant.
\end{proof}

\begin{corollary}
	$\OP^2$ is a halving space for both the $\Sp(1)$-action and the $\U(1)$-action, with fixed point set $\HP^2$.
\end{corollary}
\begin{proof}[Proof sketch]
	The projective planes $\F\PP^2$ can be obtained by gluing along the Hopf fibrations:
	$$\RP^2=D^2\coprod_{\pi_1}S^1,\quad \CP^2=D^4\coprod_{\pi_2}S^2,\quad \HP^2=D^8\coprod_{\pi_3}S^4,\quad \OP^2=D^{16}\coprod_{\pi_4}S^8$$
	and the $\Ga$-action descends to the projective planes, with fixed point set the the previous one. From the naturally occurring cell decompositions we get that each space is a halving space with the fixed point set the previous one, in particular we get that $\OP^2$ is a quaternionic halving/circle space. Alternatively, one can adapt the proof of the next section.
\end{proof}

$\OP^2$ has a description by (restricted) homogeneous coordinates as follows. The \emph{points of $\OP^2$} are triples $(a,b,c)\in \O^3$, such that at least one of them is real, modulo the relation that two such elements are equal if they differ by left multiplication by an element of $\O$. The \emph{lines of $\OP^2$} are defined similarly (denoted $(\OP^2)^*$), but now the equivalence relation is right multiplication. A point $x=(x_1,x_2,x_3)\in \OP^2$ \emph{is incident to the line} $l=(l_1,l_2,l_3)\in (\OP^2)^*$ denoted $x\in l$, if $x_1 l_1+x_2l_2+x_3l_3=0$ for representatives chosen such that at least two of the sets $\{x_i,l_i\}$ contain a real number. The \emph{flag manifold $\Fl(\O)$} can be defined as the set of incident point-lines:
$$ \Fl(\O):=\{(x,l): x\in l\}\se \OP^2\times (\OP^2)^*.$$
\begin{remark}
	The description in terms of coordinates is in fact isomorphic to the usual model of $\OP^2$ by the exceptional Jordan algebra $\hh_3(\O)$, see \cite{Freudenthal1964}, \cite{OnishchikVinberg1994}, \cite{MareWillems2013}. These identifications are due to \cite{Aslaksen1991}, \cite{Allcock1997}, see also \cite[Theorem 7.2]{Rosenfeld1997}, \cite{CsikosKiss2011}.
\end{remark}

The automorphisms of the normed algebra $\O$ fixing $\HH$ (for further details, see \cite[Proposition B.4.1.]{thesis}) induces a coordinate-wise action on $Y=\OP^2$ with fixed point set $Y^\Ga=\HP^2$. Since the action is compatible with the incidence relation, it also induces an action on $X=\Fl(\O)$ with fixed point set $X^\Ga=\Fl(\HH^3)$.

\begin{theorem}\label{thm:octoflagcirclespace}
	With the $\Ga=\Sp(1)$-action defined above, $\Fl(\O)$ is a quaternionic halving space, with fixed point set $\Fl(\HH^3)$.
	Furthermore
	$$ \ka[\si_w^\O]=[\si_w^\HH]$$
	where $w\in S_3$ and $[\si_w^\HH]\in H^{2|w|}(\Fl(\HH^3))$.
\end{theorem}
\begin{proof}
 For $d_\bullet\in \Fl(\O)^\Ga=\Fl(\HH^3)$, the flag manifold $\Fl(\O)$ has a decomposition into $\Ga$-invariant Schubert $8i$-cells $\Om^\O_w(d_\bullet)$, defined by incidence relations, whose fixed point sets are $\Om_w^\HH(d_\bullet)$. In particular, the closures of the Schubert cells are $\Ga$-invariant halving cycles $\si_w^\O(d_\bullet)$ by a dimension count. To see that the $\Sp(1)$-multiplicity of the normal space $\nu(\HH\se \O)$ equals 1, note that the normal $\Sp(1)$-representation is its defining representation, since it acts freely and transitively on $S^3$ (see \cite[Proposition B.4.1.]{thesis}).
 The conditions of the generalized Borel-Haefliger theorem for $\Ga=\Sp(1)$ have to be checked according to Remark \ref{rmk:BHI} ii). First, the Schubert cycles are $\Sp(1)$-invariant halving cycles and their fixed point sets are cycles (this is straightforward). Second, $\Fl(\O)$ satisfies the localization theorem for $\Sp(1)$ by \cite[Theorem III.3.8.]{TomDieck}, see also \cite[Theorem 1.3]{MareWillems2013}.
\end{proof}
\begin{remark}\mbox{}

	\begin{itemize}\itemsep0em
		\item[i)] By the general theory, $\Fl(\O)$ has a Bruhat cell decomposition as $N$-orbits, see \cite{DuistermaatKolkVaradarajan1983}, \cite{MareWillems2013}. This agrees with the Schubert cell decomposition---this can be verified through the Jordan algebra model of $\OP^2$.
		\item[ii)] These examples are also examples of circle spaces. Indeed, one can restrict the action of $\Ga=\Sp(1)$ to a $\De=\U(1)$-action, such that $\O^\De=\HH$, and then the rest of the proofs are the same.
	\end{itemize}
	\end{remark}

\subsection{Quaternionic toric varieties}

In the case of $\Ga=\Z_2$, smooth toric manifolds are conjugation spaces \cite[Example 8.7]{HausmannHolmPuppe2005}. This example has a generalization in the context of circle spaces; quaternionic toric varieties introduced by Scott \cite{Scott1995}, which come naturally equipped with an $\SO(3)$-action. The cohomology of a nonsingular quaternionic toric space is generated by geometric cycles of degrees $4i$ \cite[p.\ 43]{Scott1993} and is degree-doubling isomorphic to its complex counterpart \cite[Theorems 3.3.2. and 5.5.1]{Scott1993}, and therefore are circle spaces.

\subsection{Constructions} \label{subsec:constr} We can construct new halving spaces out of old ones. Most ideas of \cite{HausmannHolmPuppe2005} about conjugation spaces can be adapted, however there are some new features.

Given a $\U(1)$-action $\rho:\U(1)\to \operatorname{Homeo}(X)$ on a space $X$ we can \emph{rescale} the action by composing it with $z\mapsto z^k$. If the action is already in this form, then we can also \emph{downscale} it: For example the $\U(1)$-action on the even real flag manifolds and the quaternionic flag manifolds have the property that $-1\in \U(1)$ acts trivially, by downscaling we can get a new action with all normal weights equal to 1. It is elementary to check that rescaling a circle space provides a circle space.
With the same proof as in \cite[Proposition 4.5]{HausmannHolmPuppe2005} one can show that
\begin{proposition} Suppose that $X$ and $Y$ are halving spaces and $H^q(X;R)$ has finite rank for all $q$. Then $X\times Y$ is also a halving space. \end{proposition}
Combining product with rescaling we can construct circle spaces with prescribed normal weights. For example rescale the circle space $S^4$ of Proposition \ref{prop:s4n} with given integers, and take the product of these. With the same proof as in \cite[Proposition 4.6]{HausmannHolmPuppe2005} one can show that
\begin{proposition}\label{prop:directlimit}
	Let $(X_i,f_{ij})$ be a direct system of halving spaces which are $T_1$ and $f_{ij}$ are $\Ga$-equivariant inclusions. Then $X=\varinjlim_{i} X_i$ is a halving space with cohomology frame $(\varprojlim \ka_i, \varprojlim \si_i)$.
\end{proposition}
Conjugation manifolds of the same dimension are locally isomorphic at a fixed point. This is no longer true for circle spaces, they are locally isomorphic at a fixed point only if they have the same normal weights. So for connected sums, \cite[ Proposition 4.7]{HausmannHolmPuppe2005} has to be modified:
\begin{proposition} Suppose that $X$ and $Y$ are circle manifolds having  the same normal weights. Then the connected sum $X\# Y$ is also a circle manifold.\end{proposition}
This construction provides examples of circle manifolds which are not homogeneous manifolds.

\begin{remark}
	No complex projective variety $X$ can be a circle space; $H^2(X)$ contains the non-zero hyperplane section, which violates the condition of having nonzero cohomology groups only in degrees $4i$.
\end{remark}
\begin{remark}
	If the  homogeneous space $X=G/H$ is a circle space, then $\rk(G)=\rk(H)$. Indeed, it is classical (e.g.\ \cite{GuilleminHolmZara2006}), that the Euler-characteristic of a homogeneous space is zero if $\rk(G)>\rk(H)$, which means that $X$ has nonzero cohomology in some odd degree, again violating the condition on even degrees. Indeed, all of the homogeneous examples above satisfy this condition.
\end{remark}

\section{Applications for enumerative problems}\label{sec:applications}
One of the main applications of the cohomology ring structure of real flag manifolds concerns enumerative geometry, namely Schubert calculus. Whereas in the complex case enumerative problems are completely solved by the cohomological product of the corresponding cycles, in the real case the
product only gives a lower bound---the number of solutions depends on the given configuration of the enumerative problem.

\subsection{Real Schubert problems}
There is no general theorem (as of yet), which gives all possible solutions to a real Schubert problem, although certain special cases have been described, see e.g.\ \cite{Sottile1997}, \cite{HeinHillarSottile2013}, \cite{FeherMatszangosz2016}.

The interpretation of the cohomology ring given in Corollary \ref{cor:realLR} is a general result providing lower bounds to an infinite family of real Schubert problems, what we can call \emph{double Schubert problems}: which involves only the double Schubert varieties $\si_{DI}^\R$ defined in Section \ref{subsec:realflags}. The details and several examples, can be found in \cite[Section 7]{realflagpaper}. Let us demonstrate this technique on a simple example:

 \begin{problem}\label{prob:real} How many $W\in\Gr_8(\R^{16})$ intersect four generic $U_i\in\Gr_8(\R^{16})$ in 4 dimensions?	\end{problem}
This problem can be rewritten as computing the intersection of four Schubert varieties $\bigcap_{i=1}^4 \si_\la(U_i)$, where $\la=(4^4)$. Each point of intersection inherits a sign from the orientations, therefore the cohomological product $[\si_\la^\R]^4\in H^*(\Gr_8(\R^{16}))$ gives a lower bound to the number of solutions. By Corollary \ref{cor:realLR}, this cohomological product can be computed as $[\si_\mu^\C]^4\in H^*(\Gr_4(\C^{8}))$ where $\mu=(2^2)$, and a simple verification shows that this product equals 6. In \cite{FeherMatszangosz2016}, we showed via elementary techniques that this lower bound is sharp, furthermore that the number of possible solutions to this problem is $6,14,30,70$. However, the cohomological lower bound works more generally for any double Schubert problem, which is a significantly larger family than the one described in \cite{FeherMatszangosz2016}.

\subsection{Quaternionic Schubert problems}
We can also solve quaternionic enumerative problems:

 \begin{proposition}\label{prop:quaternionicSchubert}
 	The number of solutions of a generic quaternionic Schubert problem is the same as the corresponding complex Schubert problem.
 \end{proposition}
\begin{proof} If the Schubert varieties $\si_j(F_\bullet^\HH)$ are transverse (generic), then since the tangent spaces are canonically oriented (the tangent bundle has a complex structure), the cohomology computation gives the exact number of solutions. We can conclude by the Littlewood-Richardson coefficients of Corollary \ref{cor:quaternionicLR}.
 \end{proof}

For instance:
\begin{problem}\label{prob:quat} How many lines intersect four given lines in $\HH P^3$?	\end{problem}

By Proposition \ref{prop:quaternionicSchubert}, the answer is the same as in the complex case, which is 2. Notice that by forgetting the $\Sp(1)$-structure and retaining only the real linear structure, we obtain the real problem discussed above, which has at least 6 solutions. These solutions are $\Sp(1)$-invariant, so they should be quaternionic. The resolution of this seeming contradiction lies in the subtle notion of genericity: the real problem obtained by forgetting the quaternionic structure is usually \emph{not real generic}, so these cases are not covered by Proposition \ref{prop:quaternionicSchubert}.

On this example, nongenericity can be seen explicitly as follows. To a given configuration $U_1\stb U_4\in \Gr_8(\R^{16})$, one can associate a real linear map $\varphi:U_1\to U_1$ with the property that the problem is real generic iff all eigenvalues of this map are distinct, and different from 0 and 1, see \cite[Remark 4.14]{FeherMatszangosz2016}. If the $U_i$ are complex ($\U(1)$-invariant), the corresponding map $\varphi:U_1\to U_1$ is complex linear. In case the eigenvalues of $\varphi$ as a complex map are distinct and contain no complex conjugate pairs, then the eigenvalues of $\varphi$ as a real map are also distinct, and the problem is also real generic. However, if the $U_i$ are quaternionic ($\Sp(1)$-invariant) then $\varphi$ is quaternionic linear. Quaternionic linear maps can be written in the form
      $$ \GL_n(\HH)=\left\{\left(\begin{array}{cc}  A& B\\ -\conj{B}&\conj{A}  \end{array}\right)\in \GL_{2n}(\C)\right\}$$
and their eigenvalues as a real linear map come in fours $(\la,\la,\conj{\la}, \conj{\la})$. Since these eigenvalues are not distinct, the problem is not real generic, and it is also not hard to see that it has infinitely many solutions (since these are in bijection with invariant subspaces of $\varphi$, \cite[Corollary 2.4]{FeherMatszangosz2016}).
\section{Further examples of halving cycles} \label{sec:further}
As we explained in the previous chapter, every halving cycle in an even real flag manifold can lead to a lower bound for a corresponding real enumerative question. In the case of conjugation spaces, there is a large class of examples of halving cycles, namely the complexified subvarieties. However, there is no trivial analogue of the complexification operation for $\U(1)$-actions, so it is not easy to find non-trivial examples of halving cycles in a circle space. For even flag manifolds we have the even Schubert varieties and more generally Richardson varieties.

Below we discuss a less obvious class of examples, obtained by using quivers. We also indicate the enumerative consequences.
\subsection{Universal degeneracy loci}
Let $X$ be an even real partial flag manifold with the $\Gamma$-action of Section \ref{subsec:realflags} with fixed point set the complex partial flag manifold $X^\Gamma$. Let $\ga\in \N^n$ be a dimension vector, and $\underline{E}=(E_1,E_2,\dots,E_n)$ be an $n$-tuple of real $\Ga$-equivariant vector bundles of rank $2\gamma_i$ over $X$ such that $E_i|_{X^\Ga}$ has the structure of a complex vector bundle on which $\Gamma$ acts by scalar multiplication. We can construct such bundles using the various tautological subbundles over the flag manifold. Notice that $\Gamma$ acts on the bundles $\Hom_\R(E_i,E_{i+1})$  via conjugation with fixed point set $\Hom_\C(E_i|_{X^\Gamma},E_{i+1}|_{X^\Gamma})$. Therefore the total space of the type $A_n$ real quiver bundle with dimension vector $2\ga=(2\ga_1\stb 2\ga_n)$
\[ Q_\R(\underline{E}):=\bigoplus_{i=1}^{n-1}\Hom_\R(E_i,E_{i+1})\]
is also a circle manifold and its fixed point set is the  total space of the type $A_n$ complex quiver bundle with dimension vector $\ga$
\[ Q_\C(\underline{E}):=\bigoplus_{i=1}^{n-1}\Hom_\C(E_i|_{X^\Gamma},E_{i+1}|_{X^\Gamma}).\]

Let $Z_{2m}$ be the orbit closure in the quiver representation space $\Rep^\R_{2\ga}$ corresponding to the module $2m$. Then we can associate to $Z_{2m}^\R$ a subset $Z_{2m}^\R(\underline{E})$ of $Q_\R(\underline{E})$, the union of '$Z_{2m}^\R$-points' in all fibers. The proof of  \cite[Theorem 5.3.6.]{thesis} implies that
\begin{proposition}\label{quiver-halving}
$Z_{2m}^\R(\underline{E})$ is a halving cycle with fixed point set $Z_{m}^\C(\underline{E}|_{X^\Gamma})$.
\end{proposition}
The proof is based on the orientability of the Reineke resolution of $Z_{2m}^\R(\underline{E})$. These quiver loci are nontrivial examples of halving cycles in $Q_\R(\underline{E})$, however they are not directly related to enumerative problems.
\subsection{A degeneracy locus in the Grassmannian}
If we want to find new halving cycles in an even real partial flag manifold $X$ (and not in a bundle over $X$), then we need a section of $Z_{2m}^\R(\underline{E})$ with nice properties:

\begin{observation} \label{trans} Let $f:X\to Y$ be a $\Gamma$-equivariant map of smooth circle manifolds, and let $Z\subset Y$ be a halving cycle. If $f$ is transversal to $Z$ and $f|_{X^\Ga}:X^\Ga\to Y^\Ga$ is transversal to $Z^\Ga$, then $f^{-1}(Z)$ is also a halving cycle.
  \end{observation}
We say that $f$ is transversal to a topological subvariety $Z$, if it is transversal to a fat  nonsingular subset $U$ of $Z$, and that $f^{-1}(U)$ is a fat nonsingular subset of $f^{-1}(Z)$.
\medskip

Suppose now that $\sigma:X\to Q_\R(\underline{E})$ is a $\Gamma$-equivariant section such that $\sigma|_{X^\Gamma}:X^\Gamma\to  Q_\C(\underline{E})$ is holomorphic and transversal to $Z_{m}^\C(\underline{E})$. If $\sigma$ is transversal to $Z_{2m}^\R(\underline{E})$, then by Observation \ref{trans}, the degeneracy locus $\sigma^{-1}(Z_{2m}^\R(\underline{E}))\subset X$ is a halving cycle.

Such $\Gamma$-equivariant maps are not easy to find. The task of finding holomorphic sections of  $Q_\C(\underline{E})$ is already quite involved, so we restrict ourselves to the example  of $X=\Gr_4(\R^8)$ with fixed point set $X^\Ga=\Gr_2(\C^4)$. Before defining these sections, let us introduce some natural subvarieties of the Grassmannian. Fix a linear map $\al:\C^4\to \C^4$ and consider the following subvariety of $\Gr_2(\C^4)$:
    $$ \Sigma_\al = \{ V: \dim V\cap \al(V)\geq 1\}.$$
By forgetting the complex structure, we obtain a real linear map $\al_\R:\R^8\to \R^8$, and can similarly define a subvariety of $\Gr_4(\R^8)$:
     $$  \Sigma_\al^\R = \{V: \dim V\cap \al_\R(V)\geq 2\}.$$
For appropriate maps $\al$, these subvarieties are halving cycles:
\begin{proposition}\label{prop:Zalpha}
	If $\al:\C^4\to\C^4$ is diagonalizable with 4 different non-real eigenvalues, containing no complex conjugate pairs, then $ \Sigma_\al^\R$ is a halving cycle, with fixed point set $ \Sigma_\al$.
\end{proposition}
These subvarieties arise as degeneracy loci as follows. Let $E_1$ be the tautological subbundle and $E_2$ be the tautological quotient bundle over $X$. Then $E_1|_{X^\Gamma}$ is the tautological subbundle and $E_2|_{X^\Gamma}$ is the tautological  quotient bundle over $\Gr_2(\C^4)$. The map $\al$ induces a holomorphic section
\[\sigma^\C_\al:X^\Ga\to \Hom(E_1|_{X^\Gamma},E_2|_{X^\Gamma}),\]
where $\sigma^\C_\al(V,v)=[\al(v)]\in \C^4/V$ for $V\in \Gr_2(\C^4)$ and $v\in V$. Similarly, $\al_\R$ also induces a $\Ga$-equivariant section
\[\sigma_\al:X\to \Hom(E_1,E_2),\]
such that $\sigma_\al|_{X^\Gamma}=\sigma^\C_\al$. Then
\[ \Sigma_\al^\R=\bar\Si^2(\si_\al)\se \Gr_4(\R^8),\]
 where $\bar\Sigma^2(\sigma_\al)$ denotes the locus where the corank of $\sigma_\al$ is at least 2. Notice that $\bar\Sigma^2$ is of the form $Z_{2m}$ for an $A_2$ quiver representation. The section $\sigma_\al$ is indeed transversal to $\bar\Sigma^2(E_1,E_2)$, however the proof is quite technical, so instead we sketch a direct proof of the fact that $\Sigma_\al^\R$ is a halving cycle.
\begin{proof}[Proof (sketch)]
The key is to show that $\Sigma_\al^\R$ is a cycle. For this we give a stratification. For an arbitrary subspace $W\leq \R^8$, let us introduce the notation $W':=W\cap \al(W)$. We partition $Y=\bar\Sigma^2(\sigma_\al)$ according to the dimension of $V'$. We have  $Y=Y_2\coprod Y_3\coprod Y_4$, where
   \[ Y_i=\{V\in \Gr_4(\R^8):\dim(V')=i\}.  \]
In order to obtain a stratification, we further partition $Y_2$; by denoting $V''=V'\cap \al(V')$, let
\[ Z_i:=\{V\in Y_2:\dim(V'')=i\} \]
for $i=0,1,2$ is a partition of $Y_2$. The stratification we consider is given by
$$ \Sigma_\al^\R=Z_0\coprod Z_1\coprod Z_2\coprod Y_3\coprod Y_4.$$
For a generic 2-dimensional subspace $W$ we have $W'=0$, so the assignment $W\mapsto \langle W,\al(W)\rangle$ identifies $Z_0$ with an open submanifold of $\Gr_2(\R^8)$: this is the open---and orientable---stratum of $Y_2$. If the eigenvalues of $\al$ satisfy the conditions of Proposition \ref{prop:Zalpha}, then the codimensions of $Z_2$, $Y_3$ and $Y_4$ are greater than 2, so we don't need to study them, and it is enough to concentrate on $Z_1$.

In the remaining part, we show that $Z_1$ is a one-codimensional non-orientable stratum, which implies that $\Sigma_\al^\R$ is a cycle. For $V\in Z_1$ we have $V''= \langle v\rangle$ for some $v\in V$; by definition $\al^{-1}(v)\in V'$. Similarly, $\al^{-2}(v)\in V$. Since $V$ is in $Z_1$, these 3 vectors have to be independent: $V$ has a basis of the form $\{v,\al^{-1}(v),\al^{-2}(v),z\}$ for some $z\in V$. The assignment $(v,z)\mapsto V$ identifies $Z_1$ with an $\R\PP^4$ bundle over an open submanifold of $\R\PP^7$: Let $U\subset \R\PP^7$ be the open subset over which the bundles $\gamma,\al^{-1}(\gamma),\al^{-2}(\gamma)$ are independent, then we have the quotient vector bundle
    \[ \xi=\R^8/\big(\gamma\oplus\al^{-1}(\gamma)\oplus\al^{-2}(\gamma)\big) \]
    over $U$ and
\[Z_1\cong \R\PP(\xi).\]
The dimension of $Z_1$ is 11, one less than the dimension of $Z_0$ ($=\dim \Gr_2(\R^8)$). Notice that the complement of $U$ in $\R\PP^7$ has codimension higher than one, so the orientability of $Z_1$ is equivalent of the vanishing of the first Stiefel-Whitney class of the virtual  vector bundle
\[\tilde\xi:= \R^8\ominus 3\gamma.\]
An elementary calculation shows that $w_1(\tilde \xi)$ is not zero, so we established that $\bar\Sigma^2(\sigma_\al)$ is a cycle.
\end{proof}

\subsection{An enumerative problem} Consider the following question: Given four generic linear maps $\alpha_i:\C^4\to\C^4$ what is the number of 2-dimensional subspaces $V$ such that $\dim(V\cap\alpha_i(V))=1$ for $i=1,\dots,4$?

The answer can be given by first calculating the cohomology class of the subvariety $\Sigma^\C_\al\subset \Gr_2(\C^4)$. Since this is a Thom-Porteous locus \cite{Porteous1971}, a short calculation using the Giambelli-Thom-Porteous formula gives that $[\Sigma^\C_\al\subset \Gr_2(\C^4)]=2c_1$, where $c_1$ is the first Chern class of the tautological bundle. Then we need to intersect four general translates, and the number of intersection points is
\[  \int_{\Gr_2(\C^4)}(2c_1)^4=32.\]
Now Corollary \ref{cor:realLR} gives the following:
\begin{proposition}\label{32} Given four generic linear map $\alpha_i:\R^8\to\R^8$  the number of 4-dimensional subspaces $V$ such that $\dim(V\cap\alpha_i(V))=2$ for $i=1,\dots,4$ is at least 32.
\end{proposition}

\subsection{Equivariant fundamental classes} \label{sec:quiver} Proposition \ref{quiver-halving} is a result for any choice of bundles $E_i$, so it can be translated to an equivariant statement. Note that the equivariant approach is in the spirit of \cite{BorelHaefliger1961} for which one of the main motivation was to establish the relationship between complex and real Thom polynomials mod 2. We will follow \cite{thesis}.

 The key technical point is the concept of a \emph{halving group}: a group $G$, on which $\Ga$ acts by automorphisms, such that the classifying space $BG$ is a halving space in a well-defined sense. It is showed in  \cite{thesis} that if $G$ is a halving group and $X$ is a halving space, then $B_GX$ is a halving space with fixed point set $B_{G^\Ga}X^\Ga$. Furthermore, if $Z\se X$ is a halving cycle which is $G$-invariant, then $\ka:H_{G}^*(X)\to H^*_{G^\Ga}(X^\Ga)$ maps $[Z]_G$ to $\la^i[Z^\Ga]_{G^\Ga}$. For instance, this result has the following application \cite[Theorem 5.3.6.]{thesis}:
\begin{theorem}\label{thm:realquiver}
	Let $Q$ be the equioriented $A_n$ quiver. Let $\ga\in\N^n$ be a dimension vector and $Z_m\se \Rep_\ga^\C$ be the closure of the $\GL_\ga^\C$-orbit corresponding to the module $m=\sum \mu_{ij} l_{ij}$, where $l_{ij},\ 1\leq i\leq j\leq n$ are the indecomposable modules corresponding to the positive roots. Then
	$$ [Z_{m}^\C\se \Rep^\C_{\ga}]_{\GL_\ga^\C}=q(c_*)\acsa [Z_{2m}^\R\se \Rep^\R_{2\ga}]_{\GL_{2\ga}^\R}=q(p_*)$$
	{in $H^*(BGL_{\ga}^\C;\Q)$ and $H^*(BGL_{2\ga}^\R;\Q)$ respectively} where $2\ga=(2\ga_1\stb 2\ga_r)$ and $2m=\sum 2\mu_{ij}l_{ij}$.
\end{theorem}
The real orbit closures $Z_{2m}^\R$ are defined analogously. It is also possible to define the subvarieties $Z_{m}^\C$ and $Z_{2m}^\R$ using rank conditions, see e.g.\ \cite[Lemma 4.1]{FeherRimanyi2002quiver}.

The proof is essentially the same as of Proposition \ref{quiver-halving}. A similar description can be given for matrix Schubert varieties, \cite[Theorem 5.4.3.]{thesis}.

\begin{remark} One of the main motivations for  Borel and Haefliger was to show that the Thom polynomial of a real singularity can be obtained by replacing Chern classes by Stiefel-Whitney classes in the Thom polynomial of the complexified singularity (\cite[Theorem 6.2.]{BorelHaefliger1961}). In \cite{rongaint} Ronga showed that a similar connection can be established for the Thom polynomials of real $\Sigma^{2i}(2a,2b)$ and complex $\Sigma^{i}(a,b)$ singularities: the Pontryagin classes should be replaced by corresponding Chern classes. Theorem \ref{thm:realquiver} is a generalization of this result. For a topological analogue of Borel and Haefliger's theorem on singularity classes, where complex conjugation is replaced by $\U(1)$-actions, see \cite[Theorem 3.2.13.]{thesis} -- this is the main ingredient for proving Theorem \ref{thm:realquiver}.
\end{remark}
\bibliographystyle{plain}
\bibliography{biblio}
\end{document}